\documentclass[a4paper,reqno,12pt]{amsart}

\usepackage{amsmath}
\usepackage{amsthm}
\usepackage{amssymb}
\usepackage{latexsym}
\usepackage{graphicx,color}
\usepackage{booktabs}

\numberwithin{equation}{section}
\numberwithin{table}{section}

\newtheorem{thm}{Theorem}[section]
\newtheorem{prop}[thm]{Proposition}
\newtheorem{lem}[thm]{Lemma}

\theoremstyle{definition}
\newtheorem{defn}[thm]{Definition}

\theoremstyle{remark}
\newtheorem{rem}[thm]{Remark}

\allowdisplaybreaks[3]

\newcommand{\al}{\alpha}

\newcommand{\ga}{\gamma}

\newcommand{\de}{\delta}

\newcommand{\e}{\varepsilon}

\newcommand{\sgm}{\sigma}
\newcommand{\ze}{\zeta}

\newcommand{\p}{\partial}
\newcommand{\med}{\mathrm{med}}
\newcommand{\I}{\infty}
\newcommand{\Sc}[1]{\mathcal{#1}}
\newcommand{\F}{\Sc{F}}

\newcommand{\FR}[1]{\mathfrak{#1}}
\newcommand{\Bo}[1]{\mathbb{#1}}
\newcommand{\R}{\Bo{R}}
\newcommand{\T}{\Bo{T}}
\newcommand{\Tg}{\Bo{T}_\ga}
\newcommand{\Zg}{\Bo{Z}_\ga}
\newcommand{\lec}{\lesssim}
\newcommand{\gec}{\gtrsim}
\newcommand{\hhat}{\widehat}
\newcommand{\bbar}{\overline}
\newcommand{\ti}{\widetilde}

\newcommand{\supp}[1]{\> \operatorname{supp}\> #1}
\newcommand{\Supp}[2]{\supp{#1}\subset #2}
\newcommand{\shugo}[1]{\{ #1\}}
\newcommand{\Shugo}[2]{\big\{ \, #1 \, \big| \, #2 \, \big\}}
\newcommand{\LR}[1]{{\langle #1 \rangle }}
\newcommand{\chf}[1]{\textbf{1}_{#1}}

\newcommand{\eq}[2]{\begin{equation} \label{#1} \begin{split} #2 \end{split} \end{equation}}
\newcommand{\eqq}[1]{\begin{equation*} \begin{split} #1 \end{split} \end{equation*}}
\newcommand{\mat}[1]{\begin{smallmatrix} #1 \end{smallmatrix}}
\newcommand{\norm}[2]{\big\| #1 \big\| _{#2}}
\newcommand{\tnorm}[2]{\| #1 \| _{#2}}
\newcommand{\hx}{\hspace{10pt}}

\newcommand{\ttfrac}[2]{\text{{\footnotesize $\frac{#1}{#2}$}}}

\newcommand{\eqs}[1]{\begin{gather*} #1 \end{gather*}}

\title[$I$-method for 2d Zakharov system]{Resonant decomposition and the $I$-method for the two-dimensional Zakharov system}
\author[Nobu Kishimoto]{Nobu Kishimoto}
\address{Department of Mathematics, Kyoto University, Kyoto 606-8502, Japan}
\email{n-kishi@math.kyoto-u.ac.jp}
\begin{document}

\begin{abstract}
The initial value problem of the Zakharov system on two-dimensional torus with general period is considered in this paper.
We apply the $I$-method with some `resonant decomposition' to show global well-posedness results for small-in-$L^2$ initial data belonging to some spaces weaker than the energy class.
We also consider an application of our ideas to the initial value problem on $\R^2$ and give an improvement of the best known result by Pecher (2012).
\end{abstract}

\maketitle

%%%%%%%%%%%%%%%%%%%%%
%%%%%%%%%%%%%%%%%%%%%
%%%%%%%%%%%%%%%%%%%%%

\section{Introduction}
We consider the initial value problem of the Zakharov system:
\begin{equation}\label{ZH}
\left\{
\begin{array}{@{\,}r@{\;}l}
i\p _tu+\Delta u&=nu,\qquad u:[-T,T]\times Z \to \Bo{C},\\
\p _t^2n-\Delta n&=\Delta (|u|^2),\qquad n:[-T,T] \times Z\to \R ,\\
(u,n,\p _tn)\big| _{t=0}&=(u_0,n_0,n_1)\in H^s\times H^r\times |\nabla |H^r.
\end{array}
\right.
\end{equation}
Here, $Z=\R^2$ or $\Tg ^2:=\R ^2/(2\pi \ga _1\Bo{Z})\times (2\pi \ga _2\Bo{Z})$ (two-dimensional torus of general period $\ga =(\ga _1,\ga _2)\in \R _+^2$).
$|\nabla |H^r$ denotes the space of all functions $f$ such that $|\nabla |^{-1}f\in H^r$.
The Zakharov system was introduced in~\cite{Z72} for a model of the \mbox{Langmuir} turbulence in unmagnetized ionized plasma; $u$ represents the slowly varying envelope of rapidly oscillating electric field, and $n$ is the deviation of ion density from its mean value.

\eqref{ZH} is described as a Hamiltonian PDE with the Hamiltonian given by
\eqq{H(u,n)(t):=\norm{\nabla u(t)}{L^2}^2+\ttfrac{1}{2}(\norm{n(t)}{L^2}^2+\norm{|\nabla |^{-1}\p _tn(t)}{L^2}^2)+\int _Zn(t,x)|u(t,x)|^2\,dx.}
Local well-posedness in the energy space $H^1\times L^2\times |\nabla |L^2$ was obtained in \cite{BC96} for $Z=\R^2$ and in \cite{K11} for $Z=\Tg^2$.
In particular, using conservation of mass and the Hamiltonian and the sharp Gagliardo-Nirenberg inequality
\[ \norm{u}{L^4(Z)}^4\le \frac{2}{\tnorm{Q}{L^2(\R ^2)}^2}\norm{u}{L^2(Z)}^2\norm{\nabla u}{L^2(Z)}^2+C\norm{u}{L^2(\Tg^2)}^4 \]
 (the last term in the right hand side is required only in the periodic case; see \cite{W83,CM08}), we have the a priori control of the energy norm of solutions in the energy class if $\tnorm{u_0}{L^2}<\tnorm{Q}{L^2(\R^2)}$, where $Q$ is the ground state of the cubic NLS on $\R^2$.
More precisely, if $\eta :=1-\tnorm{u_0}{L^2}^2/\tnorm{Q}{L^2(\R^2)}^2>0$, then we have
\eqq{\Big| \int n(t)|u(t)|^2\Big| &\le \norm{n(t)}{L^2}\norm{u(t)}{L^4}^2\le \ttfrac{1-\eta /2}{2}\norm{n(t)}{L^2}^2+\ttfrac{1}{2(1-\eta /2)}\norm{u(t)}{L^4}^4\\
&\le \ttfrac{1-\eta /2}{2}\norm{n(t)}{L^2}^2+\ttfrac{1-\eta}{1-\eta /2}\norm{\nabla u(t)}{L^2}^2+C\norm{u(t)}{L^2}^4.}
Therefore, we have the following a priori estimate
\eqq{&\ttfrac{\eta /2}{1-\eta /2}\norm{\nabla u(t)}{L^2}^2+\ttfrac{\eta}{4}\norm{n(t)}{L^2}^2+\ttfrac{1}{2}\norm{|\nabla |^{-1}\partial _tn(t)}{L^2}^2\\
&\hx \le H(u,n)(t)+C\norm{u(t)}{L^2}^4=H(u,n)(0)+C\norm{u_0}{L^2}^4}
as long as the solution $(u(t),n(t))$ exists in the energy class.
Consequently, \eqref{ZH} is globally well-posed for initial data in the energy space with $\tnorm{u_0}{L^2}<\tnorm{Q}{L^2(\R^2)}$.
In fact, the solution also exists globally for initial data in $H^1\times L^2\times H^{-1}$ with $\tnorm{u_0}{L^2}\le \tnorm{Q}{L^2(\R^2)}$ (see \cite{GM94b} for $Z=\R^2$ and \cite{KM11} for $Z=\Tg^2$).

The present article addresses the global well-posedness of \eqref{ZH} for some initial data without finite energy.
The proof will rely on the $I$-method, which was originally introduced by Colliander, Keel, Staffilani, Takaoka, and Tao to deal with the nonlinear Schr\"odinger equations and has been applied to a wide variety of nonlinear dispersive equations.
For the details of the $I$-method, we refer to \cite{CKSTT03,Taobook,CKSTT08} and references therein.

The $I$-method for the Zakharov system was initiated by Fang, Pecher, and Zhong~\cite{FPZ} for the $\R ^2$ case, who established the global well-posedness in $H^s\times L^2\times |\nabla |L^2$ with $1>s>\frac{3}{4}$.
Their estimate of the modified energy was mainly based on the Strichartz estimate for the Schr\"odinger equation and its bilinear refinement, as well as some crude estimates with the H\"older inequality and the Sobolev embedding.
It is worth noting that they did not use the scaling argument in the $I$-method; thus it was quite important for global well-posedness under the minimal regularity assumptions to obtain the  best estimate for the lower bound of local existence time in terms of the size of initial data.

Our principal aim is to apply the $I$-method in the periodic case $Z=\Tg^2$, where the local well-posedness of \eqref{ZH} below the energy space is known for $\frac{1}{2}\le s\le 1$, $r=0$ (\cite{K11}).
However, it turns out not to be trivial at all to adjust their argument to the periodic setting. 
In fact, since the dispersive effect is limited on torus, the same estimate as for $\R^2$ cannot be expected in general.
For example, the $L^4$ Strichartz estimate for the Schr\"odinger equation on $\Tg^2$ cannot hold without some loss of derivative (see \cite{B93-1,CW10}).
To obtain the best decay in the almost conservation law, we will use the sharp trilinear estimates established in \cite{K11} which control various interactions between two Schr\"odinger solutions and a wave solution.

We remark that, in \cite{FPZ}, the trilinear terms have the biggest contribution in the increment of the modified energy and force them to assume $s>\frac{3}{4}$.
To improve further, we shall introduce a new modified energy based on the concept of `resonant decomposition' (see \cite{CKSTT08}, for instance).
The trilinear terms then become harmless; in fact, we find that these terms are acceptable for wider regularity range $s>\frac{1}{2}$.
However, some portion of the quartilinear terms in the modified energy increment still has a large contribution, which will require the regularity $s>\frac{2}{3}$ even for the case of $\R^2$ if we estimate it in the same manner as \cite{FPZ}.
To control these quartilinear terms, we make more refined analysis with the Strichartz estimate for the wave equation.
At the end, we will push down the threshold to $s>\frac{9}{14}$.
\begin{thm}\label{thm_global}
Let $1>s>\frac{9}{14}$ and $r=0$.
Then, for any spatial period $\ga$, \eqref{ZH} on $\Tg^2$ is globally well-posed for initial data with $\tnorm{u_0}{L^2(\Tg ^2)}<\tnorm{Q}{L^2(\R^2)}$.
Moreover, the global solutions satisfy
\eqq{\sup _{-T\le t\le T}\Big( \norm{u(t)}{H^s}+\norm{n(t)}{L^2}+\norm{|\nabla |^{-1}\p _tn(t)}{L^2}\Big) \le C(1+T)^{\max \shugo{\frac{1-s}{2s-1},\,\frac{4(1-s)}{14s-9}}+}}
for any $T>0$, where the constant $C>0$ depends on $s$, the implicit constant in the exponent, and the size of initial data.
\end{thm}

\begin{rem}
(i) The period $\ga$ has nothing to do with our results, as in the local theory \cite{K11}.

(ii) In contrast to the nonperiodic problem, we know (\cite{K11}) that the data-to-solution map for \eqref{ZH} on $\Tg^2$ cannot be smooth (nor $C^2$) for $r<0$.
That is why we restrict our attention to the case $r=0$ in the above theorem.
Compare this to Theorem~\ref{thm_global_R2} below.
\end{rem}

Of course, these approaches are also effective for the $\R^2$ case.
Recently, Pecher \cite{P12} extended the previous result \cite{FPZ} on $\R^2$ to a wider regularity range, in $H^s\times H^r\times |\nabla |H^r$ with
\[ r\le 0,\quad s<r+1,\quad s(r+\tfrac{3}{2})>1.\]
The new ingredient was the global well-posedness with regularity for the wave data below $L^2$.
Note that even local well-posedness was not known in these regularities before.
He first established the local well-posedness of \eqref{ZH} with the operator $I$, and then applied the argument in \cite{FPZ} to obtain an almost conservation law of the modified energy.
Even for the case $r=0$ he could improve the previous threshold $s>\frac{3}{4}$ to $s>\frac{2}{3}$ by refining the analysis of the worst trilinear terms in the increment of the modified energy.
However, since he used the same modified energy as \cite{FPZ}, the trilinear terms still require the regularity $s>\frac{2}{3}$.
Therefore, it is strongly expected that his result, combined with our approaches, can be improved further.
We carry out this and obtain the following result.
\begin{thm}\label{thm_global_R2}
Let $s<1$, $r\le 0$ be such that $r\ge s-1$ and $s>\frac{9+3r}{14+8r}$.
Then, \eqref{ZH} on $\R^2$ is globally well-posed for initial data with $\tnorm{u_0}{L^2(\R^2)}<\tnorm{Q}{L^2(\R^2)}$.
Moreover, the global solutions satisfy
\eqq{\sup _{-T\le t\le T}\Big( \norm{u(t)}{H^s}+\norm{n(t)}{H^r}+&\norm{|\nabla |^{-1}\p _tn(t)}{H^r}\Big) \\
&\le C(1+T)^{\max \shugo{\frac{(1-s)(1+r)}{(2+r)s-1},\, \frac{4(1-s)(1+r)}{(14+8r)s-(9+3r)}}+}}
for any $T>0$, where the constant $C>0$ depends on $s$, $r$, the implicit constant in the exponent, and the size of initial data.
\end{thm}

\begin{rem}
(i) If we consider the particular case $r=0$, then the above result shows the global well-posedness for $1>s>\frac{9}{14}$ just as the periodic case.

(ii) See Figure~1 for the range of regularity in the theorem.
The previous result of Pecher~\cite{P12} is indicated by \begin{picture}(30,1)(0,3) %WinTpicVersion4.10
\unitlength 0.1in
\begin{picture}(  4.0000,  2.0000)(  0.2000, -2.2000)
% BOX 2 0 3 0 Black White
% 2 20 220 420 20
% 
{\color[named]{Black}{%
\special{pn 8}%
\special{pa 20 220}%
\special{pa 420 220}%
\special{pa 420 20}%
\special{pa 20 20}%
\special{pa 20 220}%
\special{pa 420 220}%
\special{fp}%
}}%
% LINE 3 0 3 0 Black White
% 10 340 220 140 20 220 220 30 30 100 220 20 140 420 180 260 20 420 60 380 20
% 
{\color[named]{Black}{%
\special{pn 4}%
\special{pa 340 220}%
\special{pa 140 20}%
\special{fp}%
\special{pa 220 220}%
\special{pa 30 30}%
\special{fp}%
\special{pa 100 220}%
\special{pa 20 140}%
\special{fp}%
\special{pa 420 180}%
\special{pa 260 20}%
\special{fp}%
\special{pa 420 60}%
\special{pa 380 20}%
\special{fp}%
}}%
\end{picture}%
\end{picture}, and the optimal corner is $A=(\frac{1}{4}(\sqrt{17}-1),\frac{1}{4}(\sqrt{17}-5))\approx (0.781,-0.219)$.
We extend it to the range \begin{picture}(30,1)(0,3) %WinTpicVersion4.10
\unitlength 0.1in
\begin{picture}(  4.0000,  2.0000)(  0.2000, -2.2000)
% BOX 2 0 3 0 Black White
% 2 20 220 420 20
% 
{\color[named]{Black}{%
\special{pn 8}%
\special{pa 20 220}%
\special{pa 420 220}%
\special{pa 420 20}%
\special{pa 20 20}%
\special{pa 20 220}%
\special{pa 420 220}%
\special{fp}%
}}%
% LINE 3 0 3 0 Black White
% 18 340 220 140 20 280 220 80 20 220 220 30 30 160 220 20 80 100 220 20 140 400 220 200 20 420 180 260 20 420 120 320 20 420 60 380 20
% 
{\color[named]{Black}{%
\special{pn 4}%
\special{pa 340 220}%
\special{pa 140 20}%
\special{fp}%
\special{pa 280 220}%
\special{pa 80 20}%
\special{fp}%
\special{pa 220 220}%
\special{pa 30 30}%
\special{fp}%
\special{pa 160 220}%
\special{pa 20 80}%
\special{fp}%
\special{pa 100 220}%
\special{pa 20 140}%
\special{fp}%
\special{pa 400 220}%
\special{pa 200 20}%
\special{fp}%
\special{pa 420 180}%
\special{pa 260 20}%
\special{fp}%
\special{pa 420 120}%
\special{pa 320 20}%
\special{fp}%
\special{pa 420 60}%
\special{pa 380 20}%
\special{fp}%
}}%
\end{picture}%
\end{picture}, and the optimal corner is $B=(\frac{1}{16}(\sqrt{201}-3),\frac{1}{16}(\sqrt{201}-19))\approx (0.699,-0.301)$.
\end{rem}
\begin{figure}
\input{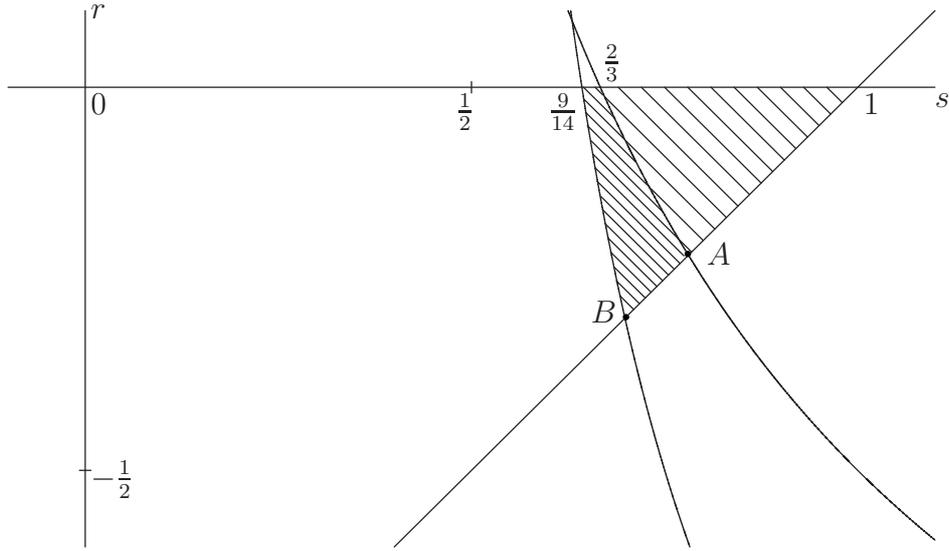}
\caption{Range of regularity for global well-posedness in the nonperiodic case.}
\end{figure}

The plan of this article is as follows.
In Section~2, we recall some definitions and estimates given in the previous results.
In Section~3, we construct our modified energy.
A proof of the almost conservation law for the periodic case and Theorem~\ref{thm_global} will be given in Section~4.
We indicate in Section~5 how to apply our ideas to the nonperiodic case, obtaining Theorem~\ref{thm_global_R2}.
In Appendix we give an elementary proof of the Strichartz estimate for the periodic wave equation, which is used in Section~4.

%%%%%%%%%%%%%%%%%%%%%%%%%%%%%%
%%%%%%%%%%%%%%%%%%%%%%%%%%%%%%
%%%%%%%%%%%%%%%%%%%%%%%%%%%%%%

\bigskip
\section{Function spaces and preliminary lemmas}

We will use the same notations as used in \cite{K11}.
\begin{defn}[Littlewood-Paley decomposition]\label{def_LPdec}
Let $\eta \in C^\I (\R )$ be an even function with the properties
\eqq{\eta \equiv 1\hx \text{on}\hx [-1,1],\quad \Supp{\eta}{(-2,2)},\quad 0\le \eta \le 1.}
Define a partition of unity on $\R$, $\eta _N$ for dyadic $N\ge 1$, by
\eqq{\eta _1:=\eta ,\qquad \eta _N(r):=\eta (\tfrac{r}{N})-\eta (\tfrac{2r}{N}),\quad N\ge 2.}
Define the frequency localization operator $P_N$ on functions $f:Z\to \Bo{C}$ by
\eqq{\F _x(P_N\phi )(\xi ):=\eta _N(|\xi |)\hhat{\phi}(\xi ).}
We also use the notation $P_N$ to denote the operator on functions in $(t,x)$,
\eqq{\F _{x}(P_Nu)(t,\xi ):=\eta _N(|\xi |)\hhat{u}(t,\xi ).}
Also, define the operators $Q_L^S$, $Q_L^{W_\pm}$ on spacetime functions by
\eqs{\F _{t,x}(Q_L^Su)(\tau ,\xi ):=\eta _L(\tau +|\xi |^2)\ti{u}(\tau ,\xi ),\quad \F _{t,x}(Q_L^{W_\pm}w)(\tau ,\xi ):=\eta _L(\tau \pm |\xi |)\ti{w}(\tau ,\xi )}
for dyadic numbers $L\ge 1$.
We will write $P^S_{N,L}=P_NQ^S_L$, $P^{W_{\pm}}_{N,L}=P_NQ^{W_{\pm}}_L$ for brevity.
Finally, we define several dyadic frequency regions:
\eqs{\FR{P}_1:=\Shugo{(\tau ,\xi )}{|\xi |\lec 2},\quad \FR{P}_N:=\Shugo{(\tau ,\xi )}{\tfrac{N}{2}\le |\xi |\le 2N},\quad N\ge 2,\\
\FR{S}_1:=\Shugo{(\tau ,\xi )}{|\tau +|\xi |^2|\lec 2},\quad \FR{S}_L:=\Shugo{(\tau ,\xi )}{\tfrac{L}{2}\le |\tau +|\xi |^2|\le 2L},\quad L\ge 2,\\
\FR{W}^\pm _1:=\Shugo{(\tau ,\xi )}{|\tau \pm |\xi ||\lec 2},\quad \FR{W}^\pm _L:=\Shugo{(\tau ,\xi )}{\tfrac{L}{2}\le |\tau \pm |\xi ||\le 2L},\quad L\ge 2.}
In what follows, capital letters $N$ and $L$ are always used to denote dyadic numbers $\ge 1$.
We will often use these capital letters with various subscripts and the notation
\eqq{\overline{N}_{ij\dots}:=\max \shugo{N_i,N_j,\dots},\qquad \underline{N}_{ij\dots}:=\min \shugo{N_i,N_j,\dots}.}
The following will be used for the specific indices;
\eqq{N_{\max}:=\overline{N}_{012},\quad N_{\min}:=\underline{N}_{012},\qquad L_{\max}:=\overline{L}_{012},\quad L_{\min}:=\underline{L}_{012}.}
\end{defn}

\begin{defn}[Function spaces $X^{s,b,p}$]
For $s,b\in \R$ and $1\le p<\I$, define the spaces $X^{s,b,p}_S$ and $X^{s,b,p}_{W_{\pm}}$ by the completion of Schwartz functions on $\R \times Z$, $Z=\R^2$ or $\Tg^2$, with respect to the following norm
\eqq{\norm{u}{X^{s,b,p}_S}&:=\norm{\norm{N^sL^b\norm{P_{N,L}^Su}{L^2_{t,x}(\R \times Z)}}{\ell ^p_L}}{\ell ^2_N},\\
\norm{u}{X^{s,b,p}_{W_{\pm}}}&:=\norm{\norm{N^sL^b\norm{P_{N,L}^{W_{\pm}}u}{L^2_{t,x}(\R \times Z)}}{\ell ^p_L}}{\ell ^2_N}.}
For $T>0$, define the restricted space $X^{s,b,p}_*(T)$ ($*=S$ or $W_{\pm}$) by the restrictions of distributions in $X^{s,b,p}_*$ to $(-T,T)\times Z$, with the norm
\eqq{\norm{u}{X^{s,b,p}_*(T)}:=\inf \Shugo{\norm{U}{X^{s,b,p}_*}}{U\in X^{s,b,p}_*~\text{is an extension of $u$ to $\R \times Z$}}.}
\end{defn}

Define the Duhamel operators
\eqs{\Sc{I}_SF(t):=-i\int _0^te^{i(t-t')\Delta}F(t')\,dt',\qquad \Sc{I}_{W_{\pm}}G(t):=i\int _0^te^{\mp i(t-t')|\nabla |}G(t')\,dt'.}
We will use a bump function $\psi _\de(t):=\psi (t/\de )$, where $\psi \in C^\I_0(\R)$ is a function with the same property as $\eta$ given in Definition~\ref{def_LPdec}.

\begin{lem}[\cite{K11}, Lemma~4.1]\label{lem_linear}
Let $s\in \R$.
For any $0<\de \le 1$ and $0<b\le \frac{1}{2}$, the following estimates hold.
The implicit constants do not depend on $s$, $\de$.
\begin{gather}
\norm{\psi _\de e^{it\Delta}u_0}{X_S^{s,\frac{1}{2},1}}\lec \norm{u_0}{H^s},\qquad \norm{\psi _\de e^{-it|\nabla |}w_0}{X_{W_+}^{s,\frac{1}{2},1}}\lec \norm{w_0}{H^s},\label{est_lin_sol}\\
\norm{\psi _\de u}{X^{s,b,1}_S}\lec \de ^{\frac{1}{2}-b}\norm{u}{X^{s,\frac{1}{2},1}_S},\qquad \norm{\psi _\de w}{X^{s,b,1}_{W_{\pm}}}\lec \de ^{\frac{1}{2}-b}\norm{w}{X^{s,\frac{1}{2},1}_{W_{\pm}}},\label{est_stability}\\
\norm{\psi _\de \Sc{I}_SF}{X^{s,\frac{1}{2},1}_S}\lec \de ^{\frac{1}{2}-b}\norm{F}{X^{s,-b,1}_S},\qquad \norm{\psi _\de \Sc{I}_{W_{\pm}}G}{X^{s,\frac{1}{2},1}_{W_{\pm}}}\lec \de ^{\frac{1}{2}-b}\norm{G}{X^{s,-b,1}_{W_+}}.\label{est_lin_duh}
\end{gather}
\end{lem}

Here and in the sequel we write $\ze =(\tau ,\xi )$.
When $Z=\Tg^2$ we use $k$ instead of $\xi$ as the discrete Fourier variable with respect to $x$ and
\eqq{\int _\ze f(\ze )=\int _{\tau \in \R}\frac{1}{\ga _1\ga _2}\sum _{k\in \Zg ^2}f(\tau ,k),\qquad \Bo{Z}_\ga ^2:=\ga _1^{-1}\Bo{Z}\times \ga _2^{-1}\Bo{Z}.}

Lemmas~\ref{lem_bs}--\ref{lem_BEforwave} below are stated for spatially periodic functions (or the Fourier transform of them) but equally hold for functions on the whole space (in this case, however, some of them are rougher than known estimates). 

\begin{lem}[\cite{K11}, Lemma~2.5 with Remark~2.8]\label{lem_bs}
Let $N_j, L_j\ge 1$ ($j=0,1,2$) be dyadic numbers.

(i) Suppose that $u_1,u_2\in L^2(\R \times \Tg ^2)$ satisfy
\eqq{\Supp{\ti{u_1}}{\FR{P}_{N_1}\cap \FR{S}_{L_1}},\qquad \Supp{\ti{u_2}}{\FR{P}_{N_2}\cap \FR{S}_{L_2}}.}
We also assume $N_0\ge 2$.
Then we have
\eqq{\norm{P_{N_0}(u_1\bbar{u_2})}{L^2_{t,x}}\lec \underline{L}_{12}^{\frac{1}{2}}\Big( \frac{\overline{L}_{12}}{N_0}+1\Big) ^{\frac{1}{2}}N_{\min}^{\frac{1}{2}}\norm{u_1}{L^2_{t,x}}\norm{u_2}{L^2_{t,x}}.}

(ii) Suppose that $u,w\in L^2(\R \times \Tg ^2)$ satisfy
\eqq{\Supp{\ti{w}}{\FR{P}_{N_0}\cap \FR{W}^{\pm}_{L_0}},\qquad \Supp{\ti{u}}{\FR{P}_{N_1}\cap \FR{S}_{L_1}}.}
Then we have
\eqq{\norm{wu}{L^2_{t,x}}+\norm{\bar{w}u}{L^2_{t,x}}\lec \underline{L}_{01}^{\frac{1}{2}}\Big( \frac{\overline{L}_{01}}{N_1}+1\Big) ^{\frac{1}{2}}\underline{N}_{01}^{\frac{1}{2}}\norm{w}{L^2_{t,x}}\norm{u}{L^2_{t,x}}.}
\end{lem}

\begin{lem}[\cite{K11}, Proposition~3.1]\label{prop_te_highlow}
Let $N_j,L_j\ge 1$ be dyadic numbers and $f,g_1,g_2\in L^2_{\ze}(\R \times \Zg ^2)$ be real-valued nonnegative functions with the support properties
\eqq{\Supp{f}{\FR{P}_{N_0}\cap \FR{W}_{L_0}^\pm},\quad \Supp{g_j}{\FR{P}_{N_j}\cap \FR{S}_{L_j}},\quad j=1,2.}
Assume $L_{\max}\gec N_{\max}^2$.
Then, we have
\eqq{\iint _{\ze _0=\ze _1-\ze _2}f(\ze _0)g_1(\ze _1)g_2(\ze _2)\lec L_{\max}^{\frac{1}{2}}L_{\med}^{\frac{1}{4}}L_{\min}^{\frac{1}{4}}N_{\min}N_{\max}^{-1}\norm{f}{L^2}\norm{g_1}{L^2}\norm{g_2}{L^2}.}
\end{lem}

\begin{lem}[\cite{K11}, Proposition~3.2]\label{prop_te_highlow2}
Let $f,g_1,g_2\in L^2_{\ze}(\R \times \Zg ^2)$ be functions as in Lemma~\ref{prop_te_highlow}, and assume $N_1\gg N_2$ or $N_2\gg N_1$.
Then, we have
\eqq{\iint _{\ze _0=\ze _1-\ze _2}f(\ze _0)g_1(\ze _1)g_2(\ze _2)\lec L_{\max}^{\frac{1}{2}}L_{\med}^{\frac{3}{8}}L_{\min}^{\frac{3}{8}}\underline{N}_{12}^{\frac{1}{2}}\overline{N}_{12}^{-1}\norm{f}{L^2}\norm{g_1}{L^2}\norm{g_2}{L^2}.}
\end{lem}

\begin{lem}[\cite{K11}, Corollary~3.4]\label{cor_te_low}
Let $f,g_1,g_2\in L^2_{\ze}(\R \times \Zg ^2)$ be functions as in Lemma~\ref{prop_te_highlow}, and assume that $N_0\lec 1$.
Then, we have
\eqq{\iint _{\ze _0=\ze _1-\ze _2}f(\ze _0)g_1(\ze _1)g_2(\ze _2)\lec (L_0L_1L_2)^{\frac{1}{6}}\norm{f}{L^2}\norm{g_1}{L^2}\norm{g_2}{L^2}.}
\end{lem}

\begin{lem}[\cite{K11}, Proposition~3.5]\label{prop_te_highhigh-m}
Let $f,g_1,g_2\in L^2_{\ze}(\R \times \Zg ^2)$ be functions as in Lemma~\ref{prop_te_highlow}.
Assume that $1\ll N_0\lec N_1\sim N_2\lec L_{\max}\ll N_1^2$.
Then, we have
\eqq{\iint _{\ze _0=\ze _1-\ze _2}f(\ze _0)g_1(\ze _1)g_2(\ze _2)\lec L_{\max}^{\frac{3}{8}+}L_{\med}^{\frac{3}{8}+}L_{\min}^{\frac{1}{4}}\big( \frac{N_0}{N_1}\big) ^{0+}\norm{f}{L^2}\norm{g_1}{L^2}\norm{g_2}{L^2}.}
\end{lem}

\begin{lem}[\cite{K11}, Proposition~3.8]\label{prop_te_highhigh-l_2d}
Let $f,g_1,g_2\in L^2_{\ze}(\R \times \Zg ^2)$ be real-valued nonnegative functions with the support properties
\eqq{\Supp{f}{\shugo{|k|\gg 1}\cap \FR{W}_{L_0}^\pm},\quad \Supp{g_j}{\FR{P}_{N_j}\cap \FR{S}_{L_j}},\quad j=1,2.}
Assume that $1\ll N_1\sim N_2$ and $L_{\max}\ll N_1$.
Then, we have
\eqq{\iint _{\ze _0=\ze _1-\ze _2}f(\ze _0)g_1(\ze _1)g_2(\ze _2)\lec L_{\max}^{\frac{3}{8}}L_{\med}^{\frac{3}{8}}\norm{f}{L^2}\norm{g_1}{L^2}\norm{g_2}{L^2}.}
\end{lem}

\begin{lem}[\cite{K11}, Proposition~4.3]\label{prop_be}
Let $\frac{1}{2}\le s\le 1$.
Then, we have
\begin{align}
\norm{\Sc{I}_S(uw)}{X^{s,\frac{1}{2},1}_S(\de )}+\norm{\Sc{I}_S(u\bar{w})}{X^{s,\frac{1}{2},1}_S(\de )}&\lec \de ^{\frac{1}{2}-}\norm{u}{X^{s,\frac{1}{2},1}_S(\de )}\norm{w}{X^{0,\frac{1}{2},1}_{W_{\pm}}(\de )},\label{est_be_s}\\
\norm{\Sc{I}_{W_+}(|\nabla |(u\bar{v}))}{X^{0,\frac{1}{2},1}_{W_+}(\de )}&\lec \de ^{\frac{1}{2}-}\norm{u}{X^{s,\frac{1}{2},1}_S(\de )}\norm{v}{X^{s,\frac{1}{2},1}_S(\de )}.\label{est_be_w}
\end{align}
\end{lem}

Next, we give a Strichartz-type estimate for the periodic (reduced) wave equation.
It seems that the Strichartz estimates in periodic setting do not follow immediately from that on the whole space, because the finite speed of propagation does not hold for the reduced wave linear propagator $e^{\mp it|\nabla|}$.
An elementary proof of it will be given in Appendix.
\begin{lem}\label{lem_BEforwave}
Let $N, L\ge 1$ be dyadic numbers, and suppose that $u\in L^2(\R \times \Tg ^2)$ satisfies $\Supp{\ti{u}}{\FR{P}_{N}\cap \FR{W}^{\pm}_{L}}$.
Then we have
\eqq{\norm{u}{L^4_{t,x}}\lec L^{\frac{3}{8}}N^{\frac{3}{8}}\norm{u}{L^2_{t,x}}.}
\end{lem}

Finally, we collect some estimates valid only for the nonperiodic case.
The next one is a refinement of Lemma~\ref{lem_bs} (ii) above.
\begin{lem}[\cite{BHHT09}, Proposition~4.3~(ii)]\label{lem_bs_R2}
Let $N_j, L_j\ge 1$ ($j=0,1$) be dyadic numbers.
Suppose that $u,w\in L^2(\R \times \R^2)$ satisfy
\eqq{\Supp{\ti{w}}{\FR{P}_{N_0}\cap \FR{W}^{\pm}_{L_0}},\qquad \Supp{\ti{u}}{\FR{P}_{N_1}\cap \FR{S}_{L_1}}.}
Then we have
\eqq{\norm{wu}{L^2_{t,x}}+\norm{\bar{w}u}{L^2_{t,x}}\lec L_0^{\frac{1}{2}}L_1^{\frac{1}{2}}\Big( \frac{\underline{N}_{01}}{N_1}\Big) ^{\frac{1}{2}}\norm{w}{L^2_{t,x}}\norm{u}{L^2_{t,x}}.}
\end{lem}

The last estimate is one of the main consequences in \cite{BHHT09}.
\begin{lem}[\cite{BHHT09}, (5.11)]\label{lem_endpoint_R2}
For smooth functions $u,w$ on $\R \times \R^2$, we have
\eq{est_endpoint_R2}{\norm{\Sc{I}_S(uw)}{X^{0,\frac{1}{2},1}_S(\de )}\lec \de ^{\frac{1}{4}}\norm{u}{X^{0,\frac{1}{2},1}_S(\de )}\norm{w}{X^{-\frac{1}{2},\frac{1}{2},1}_{W_{\pm}}(\de )}.}
\end{lem}

% Let $N_j,L_j\ge 1$ be dyadic numbers and $f,g_1,g_2\in L^2_{\ze}(\R \times \R^2)$ be real-valued nonnegative functions with the support properties
% \eqq{\Supp{f}{\FR{P}_{N_0}\cap \FR{W}_{L_0}^\pm},\quad \Supp{g_j}{\FR{P}_{N_j}\cap \FR{S}_{L_j}},\quad j=1,2.}
% Then, we have
% \eqq{\iint _{\ze _0=\ze _1-\ze _2}f(\ze _0)g_1(\ze _1)g_2(\ze _2)\lec (L_0L_1L_2)^{\frac{5}{12}}N_0^{-\frac{1}{2}}\norm{f}{L^2}\norm{g_1}{L^2}\norm{g_2}{L^2}.}

%%%%%%%%%%%%%%%%%%%%%%%%%%%%%%%%
%%%%%%%%%%%%%%%%%%%%%%%%%%%%%%%%
%%%%%%%%%%%%%%%%%%%%%%%%%%%%%%%%

\bigskip
\section{Modified energy and resonant decomposition}

In this section we introduce our almost conservation quantity and prepare some basic lemmas in the $I$-method, treating $Z=\R^2$ and $Z=\Tg^2$ at once.

With $n_+:=n+i|\nabla |^{-1}\p _tn$ and $n_{+0}:=n_0+i|\nabla |^{-1}n_1$, \eqref{ZH} is transformed into
\begin{equation}\label{ZH2}
\left\{
\begin{array}{@{\,}r@{\;}l}
i\p _tu+ \Delta u&=\tfrac{1}{2}(n_++n_-)u,\qquad u:[-T,T]\times Z\to \Bo{C},\\
i\p _tn_+-|\nabla |n_+&=|\nabla |(|u|^2),\qquad n_+:[-T,T] \times Z\to \Bo{C} ,\\
(u,n_+)\big| _{t=0}&=(u_0,n_{+0})\in H^s\times H^r,
\end{array}
\right.
\end{equation}
where $n_-:=\bbar{n_+}$, which conserves (formally) the $L^2$ norm of $u(t)$ and
\eqq{H(u,n_+)(t):=\norm{\nabla u(t)}{L^2}^2+\ttfrac{1}{2}\norm{n_+(t)}{L^2}^2+\ttfrac{1}{2}\int _{Z}(n_+(t,x)+n_1(t,x))|u(t,x)|^2dx,}
although $H(u,n_+)$ cannot be in general defined for $(u(t),n_+(t))\in H^s\times H^r$ with $s<1$ or $r<0$.
We can recover \eqref{ZH} from \eqref{ZH2} by putting $n:=\Re n_+$ since $n$ is real valued.

For $s<1$, $r\le 0$, and $N\gg 1$, we define the operator $I^S_{s,N}$ for the Schr\"odinger equation and the operator $I^{W_+}_{r,N}$ for the reduced wave equation as
\eqq{I^S_{s,N}:=\F ^{-1}_\xi m_{1-s,N}(\xi )\F _x,\qquad I^{W_+}_{r,N}:=\F ^{-1}_\xi m_{-r,N}(\xi )\F _x}
with a radial function $m_{q,N}\in C^\I (\R^2)$ ($q\ge 0$), non-increasing in $|\xi |$, such that
\eqq{m_{q,N}=\begin{cases} ~1&\text{for}\hx |\xi |<N,\\
(N/|\xi |)^q&\text{for}\hx |\xi |>2N.\end{cases}}
Note that $I^S_{s,N}\in \Sc{B}(H^s,H^1)$, $I^{W_+}_{r,N}\in \Sc{B}(H^r,L^2)$, and $I^{W_+}_{0,N}$ is the identity operator.

Define the modified energy of $(u,n_+)$ by
\eqq{&H(I^S_{s,N}u,I^{W_+}_{r,N}n_+)(t)\\
&:=\norm{\nabla I^S_{s,N}u(t)}{L^2}^2+\ttfrac{1}{2}\norm{I^{W_+}_{r,N}n_+(t)}{L^2}^2+\ttfrac{1}{2}\int _{Z}I^{W_+}_{r,N}(n_+(t,x)+n_-(t,x))|I^S_{s,N}u(t,x)|^2dx.}
The operators $I^S_{s,N}$ and $I^{W_+}_{r,N}$ only act $u$ or $\bar{u}$ and $n_\pm$, respectively, so in what follows we abbreviate as 
\eqq{H(Iu,In_+)(t):=\norm{\nabla Iu(t)}{L^2}^2+\ttfrac{1}{2}\norm{In_+(t)}{L^2}^2+\ttfrac{1}{2}\int _{Z}I(n_+(t,x)+n_-(t,x))|Iu(t,x)|^2dx.}

For an integer $p\ge 2$, we write $\int _{\Sigma _p}$ to denote
\eqq{\int _{\Sigma _p}f(\xi _1,\dots ,\xi _p):=(2\pi )^{-(p-2)}\int _{\R^2}\!\!\cdots\!\! \int _{\R^2}f(\xi _1,\dots ,\xi _p)\de (\xi _1+\dots +\xi _p=0)\,d\xi _1\cdots d\xi _p}
for the case $Z=\R^2$ and 
\eqq{\int _{\Sigma _p}f(k_1,\dots ,k_p):=(2\pi )^{-(p-2)}\cdot \frac{1}{(\ga _1\ga _2)^{p-1}}\sum _{\mat{k_1,\dots ,k_p\in \Zg ^2\\k_1+\cdots +k_p=0}}f(k_1,\dots ,k_p)}
for the case $Z=\Tg^2$.
Also, we use the notations $\xi _{ij}:=\xi _i+\xi _j$, $m_{q,j}:=m_{q,N}(\xi _j)$.
Note that
\eqq{H(Iu,In_+)=&\int _{\Sigma _2}|\xi _1|^2m_{1-s,1}^2\hhat{u}(\xi _1)\hhat{\bar{u}}(\xi _2)+\frac{1}{2}\int _{\Sigma _2}m_{-r,1}^2\hhat{n}_+(\xi _1)\hhat{n}_-(\xi _2)\\
&+\frac{1}{2}\int _{\Sigma _3}m_{1-s,1}m_{1-s,2}m_{-r,3}\hhat{u}(\xi _1)\hhat{\bar{u}}(\xi _2)\big( \hhat{n}_++\hhat{n}_-\big) (\xi _3).}

If $\tnorm{u_0}{L^2}<\tnorm{Q}{L^2(\R^2)}$, then $\tnorm{Iu(t)}{L^2}\le \tnorm{u(t)}{L^2}=\tnorm{u_0}{L^2}<\tnorm{Q}{L^2(\R^2)}$ and we have
\eqq{\norm{Iu(t)}{H^1}^2+\norm{In_+(t)}{L^2}^2\sim \norm{Iu(t)}{L^2}^2+H(Iu,In_+)(t).}
Hence, we need an almost conservation law for the modified energy, as well as the local well-posedness with the existence time written in terms of $\tnorm{Iu_0}{H^1}+\tnorm{In_{+0}}{L^2}$.
For better decay of the increment of the modified energy, we introduce another quantity
\eqq{\ti{H}(u,n_+):=&\int _{\Sigma _2}|\xi _1|^2m_{1-s,1}^2\hhat{u}(\xi _1)\hhat{\bar{u}}(\xi _2)+\frac{1}{2}\int _{\Sigma _2}m_{-r,1}^2\hhat{n}_+(\xi _1)\hhat{n}_-(\xi _2)\\
&+\frac{1}{2}\int _{\Sigma _3}\hhat{u}(\xi _1)\hhat{\bar{u}}(\xi _2)\big( \sgm _+(\xi _1,\xi _2)\hhat{n}_+(\xi _3)+\sgm _-(\xi _1,\xi _2)\hhat{n}_-(\xi _3)\big) ,}
where the multipliers $\sgm _{\pm}$ will be defined soon.
A direct calculation using the equation shows that
\eqq{&\frac{d}{dt}\ti{H}(u,n_+)\\
=&\frac{i}{2}\int _{\Sigma _3}\Big( |\xi _1|^2m_{1-s,1}^2-|\xi _2|^2m_{1-s,2}^2+|\xi _3|m_{-r,3}^2-\big( |\xi _1|^2-|\xi _2|^2+|\xi _3|\big) \sgm _+(\xi _1,\xi _2)\Big) \hhat{u}(\xi _1)\hhat{\bar{u}}(\xi _2)\hhat{n}_+(\xi _3)\\
+&\frac{i}{2}\int _{\Sigma _3}\Big( |\xi _1|^2m_{1-s,1}^2-|\xi _2|^2m_{1-s,2}^2-|\xi _3|m_{-r,3}^2-\big( |\xi _1|^2-|\xi _2|^2-|\xi _3|\big) \sgm _-(\xi _1,\xi _2)\Big) \hhat{u}(\xi _1)\hhat{\bar{u}}(\xi _2)\hhat{n}_-(\xi _3)\\
-&\frac{i}{4}\int _{\Sigma _4}\hhat{u}(\xi _1)\hhat{\bar{u}}(\xi _2)\big( \hhat{n}_++\hhat{n}_-\big) (\xi _3)\\
&\quad \times \Big( \big( \sgm _+(\xi _{13},\xi _2)-\sgm _+(\xi _1,\xi _{23})\big) \hhat{n}_+(\xi _4)+\big( \sgm _-(\xi _{13},\xi _2)-\sgm _-(\xi _1,\xi _{23})\big) \hhat{n}_-(\xi _4)\Big) \\
-&\frac{i}{2}\int _{\Sigma _4}|\xi _{12}|\big( \sgm _+-\sgm _-\big) (\xi _1,\xi _2)\hhat{u}(\xi _1)\hhat{\bar{u}}(\xi _2)\hhat{u}(\xi _3)\hhat{\bar{u}}(\xi _4).}

An initial guess for $\sgm _\pm$ would be
\eq{def_sigma_z}{\sgm _\pm (\xi _1,\xi _2)=\sgm ^Z_\pm (\xi _1,\xi _2):=\frac{|\xi _1|^2m_{1-s,1}^2-|\xi _2|^2m_{1-s,2}^2\pm |\xi _{12}|m_{-r,12}^2}{|\xi _1|^2-|\xi _2|^2\pm |\xi _{12}|},}
which kills all the trilinear terms.
Under this definition, however, $\sgm _\pm$ have singularities and we will fail to estimate the quartilinear terms.
Here arises an essential difficulty in applying the $I$-method to the Zakharov system.

In \cite{FPZ,P12}, they used
\eq{def_sigma_s}{\sgm _+(\xi _1,\xi _2)=\sgm _-(\xi _1,\xi _2)=\sgm ^S(\xi _1,\xi _2):=\frac{|\xi _1|^2m_{1-s,1}^2-|\xi _2|^2m_{1-s,2}^2}{|\xi _1|^2-|\xi _2|^2}}
so that the worst terms including two derivatives would be cancelled with $\sgm _\pm$ in the trilinear terms.
It is easy to check that $\sgm ^S$ is bounded.
However, the remaining trilinear terms are still much more massive than the quartilinear terms.
In fact, it was exactly these terms that determined the regularity threshold for global well-posedness, both in \cite{FPZ} ($s>\frac{3}{4}$) and in \cite{P12} ($s>\frac{2}{3}$).

We will use both \eqref{def_sigma_z} and \eqref{def_sigma_s} to obtain a slightly better estimate.
It turns out that the biggest contribution in the remaining trilinear terms comes from the frequency region for high-low interactions ($|\xi _1|\not\sim |\xi _2|$), which has no intersection with the region $\big| |\xi _1|^2-|\xi _2|^2\big|\sim |\xi _{12}|$, where $\sgm ^Z_\pm$ become unbounded.
Motivated by this fact, we shall employ the following definition.
\eq{def_sigma}{\sgm _\pm (\xi _1,\xi _2):=\begin{cases}
\sgm ^Z_\pm (\xi _1,\xi _2) &\text{if}\quad \big| |\xi _1|^2-|\xi _2|^2\big| >2|\xi _{12}|,\\
\sgm ^S(\xi _1,\xi _2) &\text{if}\quad \big| |\xi _1|^2-|\xi _2|^2\big| \le 2|\xi _{12}|.\end{cases}}
The above definition can be regarded as a variant of `resonant decomposition' introduced in \cite{CKSTT08} in the context of two-dimensional cubic NLS, since we consider resonant and non-resonant frequencies separately to prevent the multiplier from becoming singular.
Observe that $\sgm _\pm (\xi _1,\xi _2)=\sgm _\pm (-\xi _1,-\xi _2)=\sgm _\mp (\xi _2,\xi _1)$, and that $\sgm _\pm (\xi _1,\xi _2)\equiv 1$ when $\max \shugo{|\xi _1|,|\xi _2|}\le N/2$.
Moreover, we can easily show the following lemma.
In particular, $\sgm _\pm (\xi _1,\xi _2)$ are bounded.
\begin{lem}\label{lem_bound_sigma}
The multipliers $\sgm _\pm (\xi _1,\xi _2)$ given by \eqref{def_sigma} obey the following estimates.
\begin{enumerate}
\item If $|\xi _1|\gg |\xi _2|$, then $|\sgm _\pm (\xi _1,\xi _2)-m_{1-s,1}^2|\lec \frac{|\xi _2|^2}{|\xi _1|^2}+\frac{1}{|\xi _1|}$.
\item If $|\xi _1|\sim |\xi _2|$, then $|\sgm _\pm (\xi _1,\xi _2)|\lec 1$.
\end{enumerate}
\end{lem}

We next show that the new quantity $\ti{H}(u,n_+)$, which is our almost conserved quantity, is always close to the (first generation) modified energy $H(Iu,In_+)$.

\begin{prop}[Fixed-time difference]\label{prop_fixedtime}
Let $1>s>\frac{1}{2}$, $0\ge r>-\frac{1}{2}$.
Suppose that $r>1-2s$.
Then, for any $t\in \R$, we have
\eqq{\big| H(Iu,In_+)(t)-\ti{H}(u,n_+)(t)\big| \lec N^{-1+}\norm{Iu(t)}{H^1}^2\norm{In_+(t)}{L^2}.}
\end{prop}
\begin{proof}
From the definition and boundedness of multipliers, we have
\eqq{&\big| H(Iu,In_+)(t)-\ti{H}(u,n_+)(t)\big| \\
&\le \frac{1}{2}\int _{\Sigma _3}|\hhat{u}(\xi _1)||\hhat{\bar{u}}(\xi _2)|\Big| \big( m_{1-s,1}m_{1-s,2}m_{-r,3}-\sgm _+(\xi _1,\xi _2)\big) \hhat{n}_+(\xi _3)\\
&\hspace{120pt} +\big( m_{1-s,1}m_{1-s,2}m_{-r,3}-\sgm _-(\xi _1,\xi _2)\big) \hhat{n}_-(\xi _3)\Big| \\
&\le \frac{1}{2}\int _{\Sigma _3}\chf{\shugo{|\xi _1|>N/2~\text{or}~|\xi _2|>N/2}}(\xi _1,\xi _2)|\hhat{u}(\xi _1)||\hhat{\bar{u}}(\xi _2)|\big( |\hhat{n}_+(\xi _3)|+|\hhat{n}_-(\xi _3)|\big) .}
We may assume that all of $\hhat{u}$, $\hhat{\bar{u}}$, $\hhat{n}_\pm$ are real-valued and non-negative.
Symmetry allows us to assume $|\xi _1|\ge |\xi _2|$.
Also, it suffices to consider the case of $n_+$.
Then the above is bounded by
\eqq{&\sum _{N_1\gec N}\sum _{N_2\le N_1}\sum _{N_0\lec N_1}(\frac{N_1}{N})^{1-s}\Big( (\frac{N_2}{N})^{1-s}+1\Big) \Big( (\frac{N_0}{N})^{-r}+1\Big) \norm{P_{N_1}Iu}{L^2}\norm{P_{N_2}Iu}{L^{\I}}\norm{P_{N_0}In_+}{L^2}\\
&\lec \sum _{N_1\gec N}\sum _{N_2\le N_1}\sum _{N_0\lec N_1}(\frac{N_1}{N})^{1-s}\Big( (\frac{N_2}{N})^{1-s}+1\Big) \Big( (\frac{N_0}{N})^{-r}+1\Big) \frac{1}{N_1}\norm{P_{N_1}Iu}{H^1}\norm{P_{N_2}Iu}{H^1}\norm{P_{N_0}In_+}{L^2}.}
Since $2(1-s)-r<1$, the prefactor is exceeded by $N^{-1+}N_1^{0-}$.
Applying the Cauchy-Schwarz inequality to each summation we reach the claim.
\end{proof}

%%%%%%%%%%%%%%%%%%%%%%%%%%%%%%%%%%%%%%%%%%%
%%%%%%%%%%%%%%%%%%%%%%%%%%%%%%%%%%%%%%%%%%%
%%%%%%%%%%%%%%%%%%%%%%%%%%%%%%%%%%%%%%%%%%%

\bigskip
\section{Global solutions for the periodic case}

In this section we consider the periodic case and prove Theorem~\ref{thm_global}.
Since we always assume the wave data to be in $L^2$, the operator $I$ is operated only to the Schr\"odinger equation, so we use the notation $m(k)$ to denote $m_{1-s,N}(k)$ for simplicity.

%%%%%%%%%%%%%%%%%%%%%%%%%%%%%%%%%%%%%%%
\medskip
\subsection{Almost conservation law}

\begin{prop}[Almost conservation law]\label{prop_ac}
Let $1>s>\frac{1}{2}$, $r=0$, $0<\de \le 1$, and let $(u,n_+)$ be a smooth solution to \eqref{ZH2} on $(t,x)\in [0,\de ]\times \Tg^2$.
Then, we have
\eqq{&|\ti{H}(u,n_+)(\de )-\ti{H}(u,n_+)(0)|\lec N^{-1+}\de ^{\frac{1}{2}-}\norm{Iu}{X^{1,\frac{1}{2},1}_S(\de )}^2\norm{n_+}{X^{0,\frac{1}{2},1}_{W_+}(\de )}\\
&\hx +(N^{-2+}+N^{-\frac{5}{4}+}\de ^{\frac{1}{4}-}+N^{-1+}\de ^{1-})\big( \norm{Iu}{X^{1,\frac{1}{2},1}_S(\de )}^2\norm{n_+}{X^{0,\frac{1}{2},1}_{W_+}(\de )}^2+\norm{Iu}{X^{1,\frac{1}{2},1}_S(\de )}^4\big) .}
\end{prop}

\begin{proof}
From the definition,
\begin{align}
&\ti{H}(u,n_+)(\de )-\ti{H}(u,n_+)(0)=\int _0^\de \frac{d}{dt}\ti{H}(u,n_+)(t)\,dt\notag \\
&=\frac{i}{2}\int _0^\de \int _{\Sigma _3}\chf{\shugo{||k_1|^2-|k_2|^2|\le 2|k_{12}|}}(k_1,k_2)|k_{12}|\hhat{u}(t,k_1)\hhat{\bar{u}}(t,k_2) \label{term_ac_tri}\\
&\quad \times \Big( \big( 1-\sgm _+(k_1,k_2)\big) \hhat{n}_+(t,k_3)-\big( 1-\sgm _-(k_1,k_2)\big) \hhat{n}_-(t,k_3)\Big) \,dt\notag \\
&\hx -\frac{i}{4}\int _0^\de \int _{\Sigma _4}\hhat{u}(t,k_1)\hhat{\bar{u}}(t,k_2)\big( \hhat{n}_++\hhat{n}_-\big) (t,k_3)\label{term_ac_quar1}\\
&\quad \times \Big( \big( \sgm _+(k_{13},k_2)-\sgm _+(k_1,k_{23})\big) \hhat{n}_+(t,k_4)+\big( \sgm _-(k_{13},k_2)-\sgm _-(k_1,k_{23})\big) \hhat{n}_-(t,k_4)\Big) \,dt\notag \\
&\hx -\frac{i}{2}\int _0^\de \int _{\Sigma _4}|k_{12}|\big( \sgm _+-\sgm _-\big) (k_1,k_2)\hhat{u}(t,k_1)\hhat{\bar{u}}(t,k_2)\hhat{u}(t,k_3)\hhat{\bar{u}}(t,k_4)\,dt.\label{term_ac_quar2}
\end{align}

\textbf{\underline{Estimate of \eqref{term_ac_tri}}}.
We may assume $\max \shugo{|k_1|,|k_2|}>N$; otherwise $\eqref{term_ac_tri}=0$.
Note that $\big| |k_1|^2-|k_2|^2\big| \le 2|k_{12}|$ implies $\big| |k_1|-|k_2|\big| \le 2$.
Therefore, we may assume $|k_1|\sim |k_2|\gec N$.
We shall see only the first term in \eqref{term_ac_tri}, since the second one is exactly the complex conjugate of the first one.
Thus, we need to estimate
\eqq{&\Big| \int _\R \int _{\Sigma _3}\chf{\shugo{||k_1|^2-|k_2|^2|\le 2|k_{12}|}}(k_1,k_2)|k_{12}|\psi _\de \hhat{u}(t,k_1)\psi _\de \hhat{\bar{u}}(t,k_2)\big( 1-\sgm _+(k_1,k_2)\big) \chi _\de \hhat{n}_+(t,k_3)\,dt\Big| \\
&\lec \int _{\ze _0=\ze _1-\ze _2}\chf{\shugo{||k_1|^2-|k_2|^2-|k_0||\lec |k_0|}}|k_0||\ti{\psi _\de u}(\ze _1)\ti{\psi _\de u}(\ze _2)\ti{\chi _\de n_+}(\ze _0)|\\
&\le \sum _{N_1\sim N_2\gec N}\sum _{N_0\lec N_1}\sum _{L_0,L_1,L_2}N_0\int _{\ze _0=\ze _1-\ze _2}\Big| \big[ \ti{P^S_{N_1,L_1}\psi _\de u}\big] (\ze _1)\big[ \ti{P^S_{N_2,L_2}\psi _\de u}\big] (\ze _2)\big[ \ti{P^{W_+}_{N_0,L_0}\chi _\de n_+}\big] (\ze _0)\Big| ,}
where $\chi _\de :=\chf{[0,\de ]}$.
We remark that in the above summation, since $\big| |k_1|^2-|k_2|^2-|k_0|\big| \lec |k_0|$, either $L_{\max}\lec N_0$ or $L_{\max}\sim L_{\med}$ holds.
Then, from Lemmas~\ref{prop_te_highlow}, \ref{cor_te_low}--\ref{prop_te_highhigh-l_2d}, this is bounded by
\begin{align}
&\sum _{N_1\sim N_2\gec N}\sum _{N_0\lec N_1}\sum _{L_0,L_1,L_2}N_0(L_{\max}L_{\med})^{\frac{3}{8}+}L_{\min}^{\frac{1}{4}}\norm{P^S_{N_1,L_1}\psi _\de u}{L^2_{t,x}}\norm{P^S_{N_2,L_2}\psi _\de u}{L^2_{t,x}}\norm{P^{W_+}_{N_0,L_0}\chi _\de n_+}{L^2_{t,x}}\notag \\
&\lec \sum _{N_1\sim N_2\gec N}N_1\norm{P_{N_1}\psi _\de u}{X^{0,\frac{3}{8}+,1}_S}\norm{P_{N_2}\psi _\de u}{X^{0,\frac{3}{8}+,1}_S}\norm{\chi _\de n_+}{X^{0,\frac{1}{4},1}_{W_+}}+~\text{similar terms}\notag \\
&\lec \sum _{N_1\sim N_2\gec N}\frac{1}{N_1}(\frac{N_1}{N})^{2(1-s)}\norm{P_{N_1}\psi _\de Iu}{X^{1,\frac{3}{8}+,1}_S}\norm{P_{N_2}\psi _\de Iu}{X^{1,\frac{3}{8}+,1}_S}\norm{\chi _\de n_+}{X^{0,\frac{1}{4},1}_{W_+}}+~\text{similar terms}\label{term_1}\\
&\lec N^{-1}\norm{\psi _\de Iu}{X^{1,\frac{3}{8}+,1}_S}\norm{\psi _\de Iu}{X^{1,\frac{3}{8}+,1}_S}\norm{\chi _\de n_+}{X^{0,\frac{1}{4},1}_{W_+}}+~\text{similar terms}\notag \\
&\lec N^{-1}\de ^{\frac{1}{2}-}\norm{Iu}{X^{1,\frac{1}{2},1}_S}^2\norm{n_+}{X^{0,\frac{1}{2},1}_{W_+}}.\notag
\end{align}
In the last inequality we have used \eqref{est_stability} and
\eq{est_stability_chi}{\norm{\chi _\de n}{X^{s,b,1}}\lec \de ^{\frac{1}{2}-b}\norm{n}{X^{s,\frac{1}{2},1}},\quad 0<b<\frac{1}{2},}
which can be verified similarly to \eqref{est_stability}.

\textbf{\underline{Estimate of \eqref{term_ac_quar1}}}.
Motivated by the argument in \cite{FPZ}, we add
\eqq{\frac{i}{4}\int _0^\de \int _{\Sigma _4}\hhat{u}(t,k_1)\hhat{\bar{u}}(t,k_2)\big( \hhat{n}_++\hhat{n}_-\big) (t,k_3)\big( \hhat{n}_++\hhat{n}_-\big) (t,k_4)\cdot \big( m_{13}^2-m_{23}^2\big) \,dt=0}
to \eqref{term_ac_quar1} and consider the estimate of
\eqs{\frac{i}{4}\int _0^\de \int _{\Sigma _4}\hhat{u}(t,k_1)\hhat{\bar{u}}(t,k_2)\big( \hhat{n}_++\hhat{n}_-\big) (t,k_3)\big( \sgm _+(k_{13},k_2)-m_{13}^2-\sgm _-(k_{23},k_1)+m_{23}^2\big) \hhat{n}_+(t,k_4)\,dt,\\
\frac{i}{4}\int _0^\de \int _{\Sigma _4}\hhat{u}(t,k_1)\hhat{\bar{u}}(t,k_2)\big( \hhat{n}_++\hhat{n}_-\big) (t,k_3)\big( \sgm _-(k_{13},k_2)-m_{13}^2-\sgm _+(k_{23},k_1)+m_{23}^2\big) \hhat{n}_-(t,k_4)\,dt.}
It is then sufficient to estimate
\eqq{&\Big| \int _0^\de \int _{\Sigma _4}\big( \sgm _\pm (k_{13},k_2)-m_{13}^2\big) \hhat{u}(t,k_1)\hhat{\bar{u}}(t,k_2)\hhat{n}_\pm (t,k_3)\hhat{n}_\pm (t,k_4)\,dt\Big| \\
&\lec \int _{\ze _1+\ze _2+\ze _3+\ze _4=0}\big| \sgm _\pm (k_{13},k_2)-m_{13}^2\big| |\ti{\psi _\de u}(\ze _1)\ti{\psi _\de \bar{u}}(\ze _2)\ti{\chi _\de n_\pm}(\ze _3)\ti{\chi _\de n_\pm}(\ze _4)|\\
&\lec \sum _{N_1,\dots ,N_4\ge 1}\int _{\ze _1+\ze _2+\ze _3+\ze _4=0}\big| \sgm _\pm (k_{13},k_2)-m_{13}^2\big| \notag \\[-10pt]
&\hspace{140pt} \times |\ti{\psi _\de P_{N_1}u}(\ze _1)\ti{\psi _\de P_{N_2}\bar{u}}(\ze _2)\ti{\chi _\de P_{N_3}n_\pm}(\ze _3)\ti{\chi _\de P_{N_4}n_\pm}(\ze _4)|}
with an arbitrary choice of $\pm$.
However, since the choice of $n_{\pm}$ plays no role in the following, we consider the case $n_+$ only, and write
\eqq{\ti{u_1}:=|\ti{\psi _\de P_{N_1}u}|,\quad \ti{\bbar{u_2}}:=|\ti{\psi _\de P_{N_2}\bar{u}}|,\quad \ti{n_3}:=|\ti{\chi _\de P_{N_3}n_+}|,\quad \ti{n_4}:=|\ti{\chi _\de P_{N_4}n_+}|}
for simplicity.
We thus need to estimate
\eq{term_quar1_red}{\sum _{N_1,\dots ,N_4\ge 1}\int _{\ze _1+\ze _2+\ze _3+\ze _4=0}\big| \sgm _\pm (k_{13},k_2)-m_{13}^2\big| \ti{u_1}(\ze _1)\ti{\bbar{u_2}}(\ze _2)\ti{n_3}(\ze _3)\ti{n_4}(\ze _4).}

First, we state an estimate which will be frequently used later.
\begin{lem}
Suppose that $u$ and $n$ satisfy
\eqq{\Supp{\ti{u}}{\FR{P}_{N_1}},\qquad \Supp{\ti{n}}{\FR{P}_{N}}}
for some dyadic $N_1,N\ge 1$.
Then, for any $0<\e \ll 1$, we have
\eq{hojolemma}{\norm{un}{L^2_{t,x}}\lec \norm{u}{X^{2\e ,\frac{1}{2},1}_S}\norm{n}{X^{0,\frac{1}{2}-\e ,1}_{W_\pm}}+\norm{u}{X^{\frac{1}{2}+\e ,\frac{1}{2},1}_S}\norm{n}{X^{0,0,1}_{W_\pm}}.}
Here, the $\pm$ signs are allowed to be chosen as $(+,+)$ or $(-,-)$ only.
\end{lem}
\begin{proof}
From Lemma~\ref{lem_bs}, we have
\eqq{\norm{un}{L^2_{t,x}}\lec \norm{u}{X^{0,\frac{1}{2},1}_S}\norm{n}{X^{0,\frac{1}{2},1}_{W_\pm}}+\norm{u}{X^{\frac{1}{2},\frac{1}{2},1}_S}\norm{n}{X^{0,0,1}_{W_\pm}}.}
On the other hand, an application of the H\"older inequality shows that
\eqq{\norm{un}{L^2_{t,x}}\lec \norm{u}{L^\I _{t,x}}\norm{n}{L^2_{t,x}}\lec \norm{u}{X^{1,\frac{1}{2},1}_S}\norm{n}{L^2_{t,x}}.}
The required estimate is obtained from an interpolation between them.
\end{proof}

Let us begin to estimate \eqref{term_quar1_red}.
First of all, we remark that the multiplier $\sgm _{\pm}(k_{13},k_2)-m_{13}^2$ vanishes if $N_2,N_4\ll N$.
We consider some cases separately.

\textbf{Case 1}. $N_2\gec N_4$.
In this case we can assume $N_2\gec N$ and bound the multiplier by $1$.
Also, we see that either $N_1$ or $N_2$ has to be comparable to the biggest one among $N_j$'s.

(i) Consider the case $N_1\gec N$.
We use \eqref{hojolemma} twice to have
\eqq{\eqref{term_quar1_red}&\lec \sum _{N_1,\dots ,N_4}\norm{u_1n_3}{L^2}\norm{\bbar{u_2}n_4}{L^2}\\
&\lec \sum _{N_1,\dots ,N_4}\big( \frac{N_1}{N}\big) ^{1-s}\big( \frac{N_2}{N}\big) ^{1-s}\frac{1}{N_1N_2}\\
&\hx \times \Big( N_1^{2\e}\norm{Iu_1}{X^{1,\frac{1}{2},1}_S}\norm{n_3}{X^{0,\frac{1}{2}-\e ,1}_{W_+}}+N_1^{\frac{1}{2}+\e}\norm{Iu_1}{X^{1,\frac{1}{2},1}_S}\norm{n_3}{X^{0,0,1}_{W_+}}\Big) \\
&\hx \times \Big( N_2^{2\e}\norm{Iu_2}{X^{1,\frac{1}{2},1}_S}\norm{n_4}{X^{0,\frac{1}{2}-\e ,1}_{W_+}}+N_2^{\frac{1}{2}+\e}\norm{Iu_2}{X^{1,\frac{1}{2},1}_S}\norm{n_4}{X^{0,0,1}_{W_+}}\Big) .}

Since $s>\frac{1}{2}$, there remains $N_1^{0-}N_2^{0-}$ if we choose $\e >0$ sufficiently small.
Summing over $N_j$'s and then applying \eqref{est_stability} and \eqref{est_stability_chi}, we obtain a bound of
\eqq{(N^{-2+}+N^{-1+}\de ^{1-})\norm{Iu}{X^{1,\frac{1}{2},1}_S}^2\norm{n_+}{X^{0,\frac{1}{2},1}_{W_+}}^2.}

(ii) Consider the case $N_1\ll N$, where we may assume $N_2\gg N_1$ and $N_2$ is comparable to the max.
We further decompose the integral as 
\eq{term_tochuu1}{\sum _{N_2\gec N}\sum _{N_1\ll N}\sum _{N_3,N_4\lec N_2}\sum _{L_1,\dots ,L_4\ge 1}\int _{\ze _1+\dots +\ze _4=0}\ti{Q^S_{L_1}u_1}(\ze _1)\ti{\bbar{Q^S_{L_2}u_2}}(\ze _2)\ti{Q^{W_+}_{L_3}n_3}(\ze _3)\ti{Q^{W_+}_{L_4}n_4}(\ze _4).}
Observe that if $\ze _1+\dots +\ze _4=0$, then
\eqq{\bbar{L}_{1234}&\gec \big| (\tau _1+|k_1|^2)+(\tau _2-|k_2|^2)+(\tau _3+|k_3|)+(\tau _4+|k_4|)\big| \\
&=\big| |k_1|^2-|k_2|^2+|k_3|+|k_4|\big| \gec N_2^2.}

We begin with the case $\bbar{L}_{34}=\bbar{L}_{1234}$.
Without loss of generality we assume $L_3$ is the biggest one.
We apply the H\"older inequality and Lemma~\ref{lem_bs}~(ii) to obtain that
\eqq{&\sum _{L_1,\dots ,L_4\ge 1}\int _{\ze _1+\dots +\ze _4=0}\ti{Q^S_{L_1}u_1}(\ze _1)\ti{\bbar{Q^S_{L_2}u_2}}(\ze _2)\ti{Q^{W_+}_{L_3}n_3}(\ze _3)\ti{Q^{W_+}_{L_4}n_4}(\ze _4)\\
&\lec \sum _{L_1,\dots ,L_4\ge 1}\norm{Q^S_{L_1}u_1}{L^\I _{t,x}}\norm{Q^{W_+}_{L_3}n_3}{L^2_{t,x}}\norm{\bbar{Q^S_{L_2}u_2}Q^{W_+}_{L_4}n_4}{L^2_{t,x}}\\
&\lec \sum _{L_2,L_4\ge 1}\norm{u_1}{X_S^{1,\frac{1}{2},1}}N_2^{-1+}\norm{n_3}{X^{0,\frac{1}{2}-,1}_{W_+}}\underline{L}_{24}^{\frac{1}{2}}\Big( \frac{\bbar{L}_{24}}{N_2}+1\Big) ^{\frac{1}{2}}N_4^{\frac{1}{2}}\norm{Q^S_{L_2}u_2}{L^2_{t,x}}\norm{Q^{W_+}_{L_4}n_4}{L^2_{t,x}}\\
&\lec \Big( \frac{N_2}{N}\Big) ^{1-s}N_2^{-2+}N_4^{\frac{1}{2}}\norm{Iu_1}{X_S^{1,\frac{1}{2},1}}\norm{Iu_2}{X_S^{1,\frac{1}{2},1}}\norm{n_3}{X^{0,\frac{1}{2}-,1}_{W_+}}\Big( N_2^{-\frac{1}{2}}\norm{n_4}{X^{0,\frac{1}{2}-,1}_{W_+}}+\norm{n_4}{X^{0,0,1}_{W_+}}\Big) .
}
At the last inequality we have used $\bbar{L}_{24}^{0+}\le L_3^{0+}$.
We perform the summation in $N_j$'s and use \eqref{est_stability} and \eqref{est_stability_chi}, concluding
\eqq{\eqref{term_tochuu1}\lec (N^{-2+}+N^{-\frac{3}{2}+}\de ^{\frac{1}{2}-})\norm{Iu}{X^{1,\frac{1}{2},1}_S}^2\norm{n_+}{X^{0,\frac{1}{2},1}_{W_+}}^2.}

We next treat $\bbar{L}_{12}=\bbar{L}_{1234}\gg \bbar{L}_{34}$, which is actually the worst case.
(When $L_1$ is the max, however, we can have some better bound than obtained below.)
If $L_2$ is the max, \eqref{term_tochuu1} is bounded by
\eqq{&\sum _{N_2\gec N}\sum _{N_1\ll N}\sum _{N_3,N_4\lec N_2}\sum _{L_1,\dots ,L_4\ge 1}\norm{Q^S_{L_1}u_1}{L^\I _{t,x}}\norm{Q^S_{L_2}u_2}{L^2_{t,x}}\norm{Q^{W_+}_{L_3}n_3}{L^4_{t,x}}\norm{Q^{W_+}_{L_4}n_4}{L^4_{t,x}}.}
Now, we use the $L^4$ Strichartz estimate for wave (Lemma~\ref{lem_BEforwave}) to bound this by
\eqq{&\sum _{N_2\gec N}\sum _{N_1\ll N}\sum _{N_3,N_4\lec N_2}\norm{u_1}{X_S^{1,\frac{1}{2},1}}N_2^{-1}\norm{u_2}{X^{0,\frac{1}{2},1}_S}(N_3N_4)^{\frac{3}{8}}\norm{n_3}{X_{W_+}^{0,\frac{3}{8},1}}\norm{n_4}{X_{W_+}^{0,\frac{3}{8},1}}\\
&\lec \sum _{N_2\gec N}\sum _{N_1\ll N}\sum _{N_3,N_4\lec N_2}\Big( \frac{N_2}{N}\Big) ^{1-s}N_2^{-\frac{5}{4}}\norm{Iu_1}{X_S^{1,\frac{1}{2},1}}\norm{Iu_2}{X^{1,\frac{1}{2},1}_S}\norm{n_3}{X_{W_+}^{0,\frac{3}{8},1}}\norm{n_4}{X_{W_+}^{0,\frac{3}{8},1}}\\
&\lec N^{-\frac{5}{4}+}\de ^{\frac{1}{4}-}\norm{Iu}{X^{1,\frac{1}{2},1}_S}^2\norm{n_+}{X^{0,\frac{1}{2},1}_{W_+}}^2.}
If $L_1$ is the max, we first apply the H\"older inequality as $L^2_tL^\I _x\cdot L^\I _tL^2_x\cdot L^4_{t,x}\cdot L^4_{t,x}$ and then make a similar argument, concluding the same bound.

\textbf{Case 2}. $N_2\ll N_4$.
In this case $|k_{13}|=|k_{24}|\gg |k_2|$ in the integral \eqref{term_quar1_red}, so we use Lemma~\ref{lem_bound_sigma} (1) to replace the multiplier with $\frac{N_2^2}{N_4^2}+\frac{1}{N_4}$.
We may also assume $N_4\gec N$.

(i) The case $N_1\gec N$.
We follow the argument in Case 1 (i).
Applying \eqref{hojolemma} twice, we have
\eqq{\eqref{term_quar1_red}\lec &\sum _{N_1,\dots ,N_4}\big( \frac{N_2^2}{N_4^2}+\frac{1}{N_4}\big) \big( \frac{N_1}{N}\big) ^{1-s}\Big( \big( \frac{N_2}{N}\big) ^{1-s}+1\Big) \frac{1}{N_1N_2}\\
&\hx \times \Big( N_1^{2\e}\norm{Iu_1}{X^{1,\frac{1}{2},1}_S}\norm{n_3}{X^{0,\frac{1}{2}-\e ,1}_{W_+}}+N_1^{\frac{1}{2}+\e}\norm{Iu_1}{X^{1,\frac{1}{2},1}_S}\norm{n_3}{X^{0,0,1}_{W_+}}\Big) \\
&\hx \times \Big( N_2^{2\e}\norm{Iu_2}{X^{1,\frac{1}{2},1}_S}\norm{n_4}{X^{0,\frac{1}{2}-\e ,1}_{W_+}}+N_2^{\frac{1}{2}+\e}\norm{Iu_2}{X^{1,\frac{1}{2},1}_S}\norm{n_4}{X^{0,0,1}_{W_+}}\Big) .}
After some calculation we reach the bound with prefactor $N^{-2+}+N^{-1+}\de ^{1-}$.

(ii) The case $N_1\ll N$, where $N_3\sim N_4$ is the max.
If $N_2$ is so small that $N_2^2\lec N_4$, the multiplier is bounded by $\frac{1}{N_4}$ and we obtain
\eqq{\eqref{term_quar1_red}&\lec \sum _{N_1,\dots ,N_4}\frac{1}{N_4}\Big( \big( \frac{N_2}{N}\big) ^{1-s}+1\Big) \norm{Iu_1}{L^{2+}_tL^\I _x}\norm{Iu_2}{L^{2+}_tL^\I _x}\norm{n_3}{L^{\I -}_tL^2_x}\norm{n_4}{L^{\I -}_tL^2_x}\\
&\lec \sum _{N_1,\dots ,N_4}\frac{1}{N_4}\Big( \big( \frac{N_2}{N}\big) ^{1-s}+1\Big) \norm{Iu_1}{X^{1,0+,1}_S}\norm{Iu_2}{X^{1,0+,1}_S}\norm{n_3}{X^{0,\frac{1}{2}-,1}_{W_+}}\norm{n_4}{X^{0,\frac{1}{2}-,1}_{W_+}}\\
&\lec N^{-1+}\de ^{1-}\norm{Iu}{X^{1,\frac{1}{2},1}_S}^2\norm{n_+}{X^{0,\frac{1}{2},1}_{W_+}}^2.}
We thus assume $N_2^2\gg N_4$.
Now, we can employ the same argument as Case 1 (ii) with a minor modification exploiting the term $\frac{N_2^2}{N_4^2}$.
The bound will be $N^{-\frac{5}{4}+}\de ^{\frac{1}{4}-}$.

\textbf{\underline{Estimate of \eqref{term_ac_quar2}}}.
We bound the multiplier $\sgm _+-\sgm _-$ by $1$, and decompose each function dyadically in $k$, obtaining a bound of \eqref{term_ac_quar2}
as
\eq{term_quar2_red}{\sum _{N_1,\dots ,N_4}(N_1+N_2)\int _\R \int _{\Tg ^2}  u_1\bbar{u_2}u_3\bbar{u_4}\,dx\,dt,}
where 
\eqq{\ti{u_1}:=|\ti{\psi _\de P_{N_1}u}|,\quad \ti{\bbar{u_2}}:=|\ti{\psi _\de P_{N_2}\bar{u}}|,\quad \ti{u_3}:=|\ti{\chi _\de P_{N_3}u}|,\quad \ti{\bbar{u_4}}:=|\ti{\chi _\de P_{N_4}\bar{u}}|.}
Without loss of generality we assume $N_1\ge N_2$, which implies $N_1\gec N$; otherwise the multiplier vanishes.
We may also assume that at least two of $N_j$'s are $\gec N$

\textbf{Case 1}. Two of $N_j$'s $\ll N$.
It will be sufficient to consider the particular case $N_1,N_2\gec N\gg N_3,N_4$, where $N_1\sim N_2$ is the max.
From a H\"older argument,
\eqq{&\eqref{term_quar2_red}\lec \sum _{N_1\sim N_2\gec N\gg N_3,N_4}N_1\norm{u_1}{L^{2+}_tL^2_x}\norm{u_2}{L^{2+}_tL^2_x}\norm{u_3}{L^{\I -}_tL^\I _x}\norm{u_4}{L^{\I -}_tL^\I _x}\\
&\lec \sum _{N_1,\dots ,N_4}\big( \frac{N_1}{N}\big) ^{1-s}\big( \frac{N_2}{N}\big) ^{1-s}\frac{1}{N_1}\norm{Iu_1}{X^{1,0+,1}_S}\norm{Iu_2}{X^{1,0+,1}_S}\norm{Iu_3}{X^{1,\frac{1}{2}-,1}_S}\norm{Iu_4}{X^{1,\frac{1}{2}-,1}_S}\\
&\lec N^{-1+}\de ^{1-}\norm{Iu}{X^{1,\frac{1}{2},1}_S}^4.}

\textbf{Case 2}. More than two of $N_j$'s $\gec N$.
Prepare the following lemma.
\begin{lem}
Suppose that $u_1$ and $u_2$ satisfy
\eqq{\Supp{\ti{u_1}}{\FR{P}_{N_1}},\qquad \Supp{\ti{u_2}}{\FR{P}_{N_2}}}
for some dyadic $N_1,N_2\ge 1$.
Then, for any $0<\e \ll 1$, we have
\eq{hojolemma2}{\norm{u_1u_2}{L^2_{t,x}}\lec \overline{N}_{12}^{\e}\big( \norm{u_1}{X^{0,\frac{1}{2}-\e ,1}_S}\norm{u_2}{X^{4\e ,\frac{1}{2}-\e ,1}_S}+\norm{u_1}{X^{0,\frac{1}{2}-\e ,1}_S}\norm{u_2}{X^{\frac{1}{2}+2\e ,\e ,1}_S}\big) .}
\end{lem}
\begin{proof}
Making dyadic decompositions, we have
\eq{term_tochuu2}{\norm{u_1u_2}{L^2}=\norm{u_1\bbar{u_2}}{L^2_{t,x}}\lec \sum _{N_0\le \overline{N}_{12}}\sum _{L_1,L_2\ge 1}\norm{P_{N_0}(Q^S_{L_1}u_1\cdot \bbar{Q^S_{L_2}u_2})}{L^2}.}
We use Lemma~\ref{lem_bs} (i) for $N_0\ge 2$ and Lemma~\ref{cor_te_low} for $N_0=1$,
\eqq{\eqref{term_tochuu2}\lec \sum _{N_0\le \overline{N}_{12}}\sum _{L_1,L_2\ge 1}\underline{L}_{12}^{\frac{1}{2}}\big( \overline{L}_{12}^{\frac{1}{2}}+N_2^{\frac{1}{2}}\big) \norm{Q^S_{L_1}u_1}{L^2}\norm{Q^S_{L_2}u_2}{L^2}.}
On the other hand, we apply the H\"older inequality to obtain
\eqq{\eqref{term_tochuu2}\lec \sum _{N_0\le \overline{N}_{12}}\sum _{L_1,L_2\ge 1}\underline{L}_{12}^{\frac{1}{2}}N_2\norm{Q^S_{L_1}u_1}{L^2}\norm{Q^S_{L_2}u_2}{L^2}.}
The required estimate is obtained from an interpolation between them.
\end{proof}

We go back to the estimate of \eqref{term_quar2_red}.
Define the biggest, the second biggest and the smallest one among $N_2,N_3,N_4$ as $N_a,N_b$ and $N_c$, respectively.
Then, we may assume that $N_a\gec N_1,N_b,N_c$.
From \eqref{hojolemma2}, we obtain
\eqq{&\eqref{term_quar2_red}\lec \sum _{N_1,N_a,N_b\gec N,N_c}N_1\norm{u_1u_b}{L^2}\norm{u_au_c}{L^2}\\
&\lec \sum _{N_1,\dots ,N_4}\big( \frac{N_1}{N}\big) ^{1-s}\big( \frac{N_b}{N}\big) ^{1-s}N_a^{\e}\big( \frac{N_a}{N}\big) ^{1-s}\Big( \big( \frac{N_c}{N}\big) ^{1-s}+1\Big) \frac{N_a^{\e}}{N_a}\\
&\hx \times \Big( N_b^{-1+4\e}\norm{Iu_1}{X^{1,\frac{1}{2}-\e ,1}_S}\norm{Iu_b}{X^{1,\frac{1}{2}-\e ,1}_S}+N_b^{-\frac{1}{2}+2\e}\norm{Iu_1}{X^{1,\frac{1}{2}-\e ,1}_S}\norm{Iu_b}{X^{1,\e ,1}_S}\Big) \\
&\hx \times \Big( N_c^{-1+4\e}\norm{Iu_a}{X^{1,\frac{1}{2}-\e ,1}_S}\norm{Iu_c}{X^{1,\frac{1}{2}-\e ,1}_S}+N_c^{-\frac{1}{2}+2\e}\norm{Iu_a}{X^{1,\frac{1}{2}-\e ,1}_S}\norm{Iu_c}{X^{1,\e ,1}_S}\Big) .}
We observe
\eqs{\big( \frac{N_1}{N}\big) ^{1-s}\big( \frac{N_a}{N}\big) ^{1-s}\frac{N_a^{2\e}}{N_a}\lec N^{-1+}N_a^{0-},\\
\big( \frac{N_b}{N}\big) ^{1-s}N_b^{-1+4\e}\lec N^{-1+},\quad \big( \frac{N_b}{N}\big) ^{1-s}N_b^{-\frac{1}{2}+2\e}\lec N^{-\frac{1}{2}+},\\
\Big( \big( \frac{N_c}{N}\big) ^{1-s}+1\Big) N_c^{-1+4\e}+\Big( \big( \frac{N_c}{N}\big) ^{1-s}+1\Big) N_c^{-\frac{1}{2}+2\e}\lec 1}
if $s>\frac{1}{2}$, and $\e >0$ sufficiently small.
Consequently, we obtain a bound of
\eqq{(N^{-2+}+N^{-\frac{3}{2}+}\de ^{\frac{1}{2}-})\norm{Iu}{X^{1,\frac{1}{2},1}_S}^4.}
Now, the proof of Proposition~\ref{prop_ac} is completed.
\end{proof}

%%%%%%%%%%%%%%%%%%%%%%%%%%
\medskip
\subsection{Modified local well-posedness}

We can upgrade the bilinear estimates given in Lemma~\ref{prop_be} to the following.
\begin{lem}\label{lem_modifiedbe}
Let $1>s\ge \frac{1}{2}$.
Then, we have
\eqs{\norm{\Sc{I}_S(I(n_{\pm}u))}{X^{1,\frac{1}{2},1}_S(\de )}\lec \de ^{\frac{1}{2}-}\norm{Iu}{X^{1,\frac{1}{2},1}_S(\de )}\norm{n_{\pm}}{X^{0,\frac{1}{2},1}_{W_{\pm}}(\de )},\\
\norm{\Sc{I}_{W_+}(|\nabla |(u_1\bbar{u_2}))}{X^{0,\frac{1}{2},1}_{W_{+}}(\de )}\lec \de ^{\frac{1}{2}-}\norm{Iu_1}{X^{1,\frac{1}{2},1}_S(\de )}\norm{Iu_2}{X^{1,\frac{1}{2},1}_S(\de )}.}
\end{lem}
\begin{proof}
The second estimate immediately follows from \eqref{est_be_w}, since $\tnorm{u}{X^{s,\frac{1}{2},1}_S}\le \tnorm{Iu}{X^{1,\frac{1}{2},1}_S}$.

For the first estimate, we decompose $u$ into two parts.
For the low frequency part, $\Supp{\ti{u}}{\shugo{|k|\lec N}}$, the claim follows from $I\le 1$ and \eqref{est_be_s} with $s=1$.
For high frequency $\Supp{\ti{u}}{\shugo{|k|\gec N}}$, we observe that
\eqq{m(k_1)\LR{k_1}^{1-s}\lec N^{1-s}\sim m(k_2)\LR{k_2}^{1-s}}
for $|k_2|\gec N$, where $k_1$ and $k_2$ denote the frequency variables for $n_{\pm}u$ and $u$, respectively.
Then the estimate follows from \eqref{est_be_s}.
\end{proof}

The standard iteration argument using Lemma~\ref{lem_modifiedbe} and \eqref{est_lin_sol} yields the modified local well-posedness adapted to the $I$-method.
\begin{prop}\label{prop_modifiedlwp}
Let $1>s\ge \frac{1}{2}$.
Then, for any $(u_0,n_{+0})\in H^s\times L^2$, there exists a unique solution to \eqref{ZH2}, $(u,n_+)\in X^{s,\frac{1}{2},1}_S(\de )\times X^{0,\frac{1}{2},1}_{W_+}(\de )$, with the existence time
\eqq{\de \sim (\norm{Iu_0}{H^1}+\norm{n_{+0}}{L^2})^{-2-},}
such that the following estimate holds:
\eqq{\norm{Iu}{X^{1,\frac{1}{2},1}_S(\de )}+\norm{n_+}{X^{0,\frac{1}{2},1}_{W_+}(\de )}\lec \norm{Iu_0}{H^1}+\norm{n_{+0}}{L^2}.}
In particular, we have
\eqq{\sup _{-\de \le t\le \de}\big( \norm{Iu(t)}{H^1}+\norm{n_+(t)}{L^2}\big) \lec \norm{Iu_0}{H^1}+\norm{n_{+0}}{L^2}.}
\end{prop}

%%%%%%%%%%%%%%%%%%%%%%%%%%%%%%%%%%%%%%
\medskip
\subsection{Proof of Theorem~\ref{thm_global}}

Let $(u_0,n_{0+})\in H^s \times L^2$ be an initial datum with $\tnorm{u_0}{L^2}<\tnorm{Q}{L^2(\R^2)}$.
The datum then satisfies
\eqq{\tnorm{Iu_0}{H^1}+\tnorm{n_{+0}}{L^2}\lec N^{1-s},\qquad \tnorm{Iu_0}{L^2}\le \tnorm{u_0}{L^2}<\tnorm{Q}{L^2(\R^2)},}
and its modified energy obeys
\eqq{H(Iu_0,n_{+0})\le C_0N^{2(1-s)}.}
Since $H(Iu,n_+)(t)$ and the (a priori bounded) $L^2$ norm of $Iu(t)$ control $\tnorm{Iu(t)}{H^1}+\tnorm{n_+(t)}{L^2}$, we see from Proposition~\ref{prop_modifiedlwp} that the solution to the initial value problem on $[0,t_0]$ can be extended up to $t=t_0+\de$ with a uniform time $\de \sim N^{-2(1-s)-}$ and satisfies
\eqq{\norm{Iu(\cdot -t_0)}{X^{1,\frac{1}{2},1}_S(\de )}+\norm{n_+(\cdot -t_0)}{X^{0,\frac{1}{2},1}_{W_+}(\de )}\lec N^{1-s},}
as long as 
\eqq{H(Iu,n_{+})(t_0)\le 2C_0N^{2(1-s)}.}
If we could iterate the local theory $M$ times, then Propositions~\ref{prop_fixedtime} and \ref{prop_ac} imply that the increment of the modified energy would be bounded by
\eqq{&|H(Iu,n_+)(M\de )-H(Iu,n_+)(0)|\\
&\le |H(Iu,n_+)(M\de )-\ti{H}(u,n_+)(M\de )|+\sum _{j=0}^{M-1}|\ti{H}(u,n_+)((j+1)\de )-\ti{H}(u,n_+)(j\de )|\\
&\hx +|\ti{H}(u,n_+)(0)-H(Iu,n_+)(0)|\\
&\lec N^{-1+}(N^{1-s})^3+M\Big\{ N^{-1+}\de ^{\frac{1}{2}-}(N^{1-s})^3+(N^{-2+}+N^{-\frac{5}{4}+}\de ^{\frac{1}{4}-}+N^{-1+}\de ^{1-})(N^{1-s})^4\Big\} \\
&\sim \Big\{ N^{-s+}+M\big( N^{-1+}+N^{\frac{1}{4}-\frac{3}{2}s+}\big) \Big\} N^{2(1-s)},}
which means that we can repeat $O(N^{\min \shugo{1,\,\frac{3}{2}s-\frac{1}{4}}-})$ times, obtaining the solution up to some time $\sim \de N^{\min \shugo{1,\,\frac{3}{2}s-\frac{1}{4}}-}\sim N^{\min \shugo{2s-1,\,\frac{7}{2}s-\frac{9}{4}}-}$.
Hence, we can solve the equation up to the arbitrarily large given time $T$ by setting a large parameter $N$ to be $\sim T^{\max \shugo{\frac{1}{2s-1},\,\frac{4}{14s-9}}+}$, whenever $s>\frac{9}{14}$.

Moreover, we have
\eqq{\sup _{-T\le t\le T}\big( \norm{u(t)}{H^s}+\norm{n_+(t)}{L^2}\big) &\lec \sup _{-T\le t\le T}\big( \norm{Iu(t)}{H^1}+\norm{n_+(t)}{L^2}\big) \\
&\lec N^{1-s}\sim T^{\max \shugo{\frac{1-s}{2s-1},\,\frac{4(1-s)}{14s-9}}+}.}
Going back to the original Zakharov system \eqref{ZH}, we obtain the a priori estimate
\eqq{\sup _{-T\le t\le T}\big( \norm{u(t)}{H^s}+\norm{n(t)}{L^2}+\norm{|\nabla |^{-1}\p _tn(t)}{L^2}\big) \lec T^{\max \shugo{\frac{1-s}{2s-1},\,\frac{4(1-s)}{14s-9}}+},}
concluding the proof of Theorem~\ref{thm_global}.

%%%%%%%%%%%%%%%%%%%%%%%%%%%%%%%%
%%%%%%%%%%%%%%%%%%%%%%%%%%%%%%%%
%%%%%%%%%%%%%%%%%%%%%%%%%%%%%%%%

\bigskip
\section{Global solutions for the nonperiodic case}

In this section we treat the $\R^2$ case and also put the operator $I$ on the wave equation.

\medskip
\subsection{Almost conservation law}
An adaptation of the argument for periodic problem easily implies the following almost conservation law.
\begin{prop}[Almost conservation law]\label{prop_ac_R2}
Let $1>s>\frac{1}{2}$, $0\ge r\ge s-1$ be such that $r>1-2s$ and $r>-\frac{1}{2}s$.
Let $0<\de \le 1$ and $(u,n_+)$ be a smooth solution to \eqref{ZH2} on $(t,x)\in [0,\de ]\times \R^2$.
Then, we have
\eqq{&|\ti{H}(u,n_+)(\de )-\ti{H}(u,n_+)(0)|\lec N^{-1+}\de ^{\frac{1}{2}-}\norm{Iu}{X^{1,\frac{1}{2},1}_S(\de )}^2\norm{In_+}{X^{0,\frac{1}{2},1}_{W_+}(\de )}\\
&\hx +(N^{-2+}+N^{-\frac{5}{4}+}\de ^{\frac{1}{4}-}+N^{-1+}\de ^{1-})\big( \norm{Iu}{X^{1,\frac{1}{2},1}_S(\de )}^2\norm{In_+}{X^{0,\frac{1}{2},1}_{W_+}(\de )}^2+\norm{Iu}{X^{1,\frac{1}{2},1}_S(\de )}^4\big) .}
\end{prop}

\begin{proof}
We follow the proof of Proposition~\ref{prop_ac} and only indicate the difference from it.
We have to consider the following three terms:
\begin{align}
&\frac{i}{2}\int _0^\de \int _{\Sigma _3}\chf{\shugo{||\xi _1|^2-|\xi _2|^2|\le 2|\xi _{12}|}}(\xi _1,\xi _2)|\xi _{12}|\hhat{u}(t,\xi _1)\hhat{\bar{u}}(t,\xi _2) \label{term_ac_tri_R2}\\
&\quad \times \Big( \big( m_{-r,12}^2-\sgm _+(\xi _1,\xi _2)\big) \hhat{n}_+(t,\xi _3)-\big( m_{-r,12}^2-\sgm _-(\xi _1,\xi _2)\big) \hhat{n}_-(t,\xi _3)\Big) \,dt,\notag \\
&-\frac{i}{4}\int _0^\de \int _{\Sigma _4}\hhat{u}(t,\xi _1)\hhat{\bar{u}}(t,\xi _2)\big( \hhat{n}_++\hhat{n}_-\big) (t,\xi _3)\label{term_ac_quar1_R2}\\
&\quad \times \Big( \big( \sgm _+(\xi _{13},\xi _2)-\sgm _+(\xi _1,\xi _{23})\big) \hhat{n}_+(t,\xi _4)+\big( \sgm _-(\xi _{13},\xi _2)-\sgm _-(\xi _1,\xi _{23})\big) \hhat{n}_-(t,\xi _4)\Big) \,dt,\notag \\
&-\frac{i}{2}\int _0^\de \int _{\Sigma _4}|\xi _{12}|\big( \sgm _+-\sgm _-\big) (\xi _1,\xi _2)\hhat{u}(t,\xi _1)\hhat{\bar{u}}(t,\xi _2)\hhat{u}(t,\xi _3)\hhat{\bar{u}}(t,\xi _4)\,dt.\label{term_ac_quar2_R2}
\end{align}

\textbf{\underline{Estimate of \eqref{term_ac_tri_R2}}}.
We bound the multiplier by $1$ as in the periodic case.
We should consider
\eqq{\sum _{N_1\sim N_2\gec N}\sum _{N_0\lec N_1}&\frac{N_0}{N_1^2}\big( \frac{N_1}{N}\big) ^{2(1-s)}\Big( \big( \frac{N_0}{N}\big) ^{-r}+1\Big) \\
&\times \norm{P_{N_1}\psi _\de Iu}{X^{1,\frac{3}{8}+,1}_S}\norm{P_{N_2}\psi _\de Iu}{X^{1,\frac{3}{8}+,1}_S}\norm{P_{N_0}\chi _\de In_+}{X^{0,\frac{1}{4},1}_{W_+}}}
instead of \eqref{term_1}.
This is bounded by $N^{-1+}\de ^{\frac{1}{2}-}\tnorm{Iu}{X_S^{1,\frac{1}{2},1}}^2\tnorm{In_+}{X_{W_+}^{0,\frac{1}{2},1}}$ in the same manner, provided $2(1-s)-r<1$.

\textbf{\underline{Estimate of \eqref{term_ac_quar1_R2}}}.
We can obtain simpler estimate 
\eqq{\norm{un}{L^2_{t,x}}\lec \norm{u}{X^{2\e ,\frac{1}{2},1}_S}\norm{n}{X^{0,\frac{1}{2}-\e ,1}_{W_\pm}}}
instead of \eqref{hojolemma} by using Lemma~\ref{lem_bs_R2} instead of Lemma~\ref{lem_bs}.

\textbf{Case 1} ($N_2\gec N_4$).

(i) $N_1\gec N$.
In this case we need to consider the quantity
\eqq{&\sum _{N_1,N_2\gec N}\sum _{N_3\lec \bbar{N}_{12}}\sum _{N_4\lec N_2}\big( \frac{N_1}{N}\big) ^{1-s}\big( \frac{N_2}{N}\big) ^{1-s}\Big( \big( \frac{N_3}{N}\big) ^{-r}+1\Big) \Big( \big( \frac{N_4}{N}\big) ^{-r}+1\Big) \frac{1}{N_1N_2}\\
&\hx \times \Big( N_1^{2\e}\norm{Iu_1}{X^{1,\frac{1}{2},1}_S}\norm{In_3}{X^{0,\frac{1}{2}-\e ,1}_{W_+}}\Big) \Big( N_2^{2\e}\norm{Iu_2}{X^{1,\frac{1}{2},1}_S}\norm{In_4}{X^{0,\frac{1}{2}-\e ,1}_{W_+}}\Big) .}
Considering the worst case $N\lec N_1\ll N_3\sim N_4\sim N_2$, we can bound the above by $N^{-2+}\tnorm{Iu}{X_S^{1,\frac{1}{2},1}}^2\tnorm{In_+}{X_{W_+}^{0,\frac{1}{2},1}}^2$ provided $1-s-2r<1$.

(ii) $N_1\ll N$.
Make the same decomposition as \eqref{term_tochuu1}.
When $\bbar{L}_{34}=\bbar{L}_{1234}$, we use Lemma~\ref{lem_bs_R2} instead of Lemma~\ref{lem_bs} to obtain the following bound,
\eqq{\sum _{N_1\ll N}\sum _{N_2\gec N}\sum _{N_3,N_4\lec N_2}&\big( \frac{N_2}{N}\big) ^{1-s}\Big( \big( \frac{N_3}{N}\big) ^{-r}+1\Big) \Big( \big( \frac{N_4}{N}\big) ^{-r}+1\Big) N_2^{-2+}\\
&\times \norm{Iu_1}{X_S^{1,\frac{1}{2},1}}\norm{Iu_2}{X_S^{1,\frac{1}{2},1}}\norm{In_3}{X^{0,\frac{1}{2}-,1}_{W_+}}\norm{In_4}{X^{0,\frac{1}{2}-,1}_{W_+}}.
}
Even the worst case $N_2\sim N_3\sim N_4\gec N$ can be estimated with decay factor $N^{-2+}$ whenever $1-s-2r<2$.
When $\bbar{L}_{12}=\bbar{L}_{1234}\gg \bbar{L}_{34}$, we follow the argument for periodic case precisely to encounter the quantity
\eqq{\sum _{N_1\ll N}\sum _{N_2\gec N}\sum _{N_3,N_4\lec N_2}&\big( \frac{N_2}{N}\big) ^{1-s}\Big( \big( \frac{N_3}{N}\big) ^{-r}+1\Big) \Big( \big( \frac{N_4}{N}\big) ^{-r}+1\Big) N_2^{-\frac{5}{4}}\\
&\times \norm{Iu_1}{X_S^{1,\frac{1}{2},1}}\norm{Iu_2}{X^{1,\frac{1}{2},1}_S}\norm{In_3}{X_{W_+}^{0,\frac{3}{8},1}}\norm{In_4}{X_{W_+}^{0,\frac{3}{8},1}}.}
This can be treated appropriately if $1-s-2r<\frac{5}{4}$.
The decay $N^{-\frac{5}{4}+}\de ^{\frac{1}{4}-}$ is obtained.

\textbf{Case 2} ($N_2\ll N_4$).

(i) $N_1\gec N$.
With a modification of the argument for periodic case similar to Case~1~(i), we estimate
\eqq{\sum _{N_1\gec N}\sum _{N_4\gec N}\sum _{N_3\lec \bbar{N}_{14}}\sum _{N_2\ll N_4}&\big( \frac{N_2^2}{N_4^2}+\frac{1}{N_4}\big) \big( \frac{N_1}{N}\big) ^{1-s}\Big( \big( \frac{N_2}{N}\big) ^{1-s}+1\Big) \Big( \big( \frac{N_3}{N}\big) ^{-r}+1\Big) \big( \frac{N_4}{N}\big) ^{-r}\frac{1}{N_1N_2}\\
\times &\Big( N_1^{2\e}\norm{Iu_1}{X^{1,\frac{1}{2},1}_S}\norm{In_3}{X^{0,\frac{1}{2}-\e ,1}_{W_+}}\Big) \Big( N_2^{2\e}\norm{Iu_2}{X^{1,\frac{1}{2},1}_S}\norm{In_4}{X^{0,\frac{1}{2}-\e ,1}_{W_+}}\Big) .}
The worst case is $N\lec N_1\ll N_2\ll N_3\sim N_4$, which is controlled if $1-s-2r<1$.
We obtain the decay $N^{-2+}$ in this case.

(ii) $N_1\ll N$.
If $N_2^2\lec N_4$, then we have
\eqq{\sum _{N_1\ll N}\sum _{N_4\gec N}\sum _{N_3\sim N_4}\sum _{N_2\lec N_4^{1/2}}&\frac{1}{N_4}\Big( \big( \frac{N_2}{N}\big) ^{1-s}+1\Big) \big( \frac{N_3}{N}\big) ^{-r}\big( \frac{N_4}{N}\big) ^{-r}\\
&\times \norm{Iu_1}{X^{1,0+,1}_S}\norm{Iu_2}{X^{1,0+,1}_S}\norm{In_3}{X^{0,\frac{1}{2}-,1}_{W_+}}\norm{In_4}{X^{0,\frac{1}{2}-,1}_{W_+}},}
which is estimated with decay $N^{-1+}\de ^{1-}$ whenever $\frac{1}{2}(1-s)-2r<1$.
If $N_2^2\gg N_4$, we can employ the same argument as Case 1 (ii) and obtain the decay $N^{-\frac{5}{4}+}\de ^{\frac{1}{4}-}$.

\textbf{\underline{Estimate of \eqref{term_ac_quar2_R2}}}.
This is identical with the periodic case, because \eqref{term_ac_quar2_R2} includes no $n_+$.
We have the bound $(N^{-2+}+N^{-1+}\de ^{1-})\tnorm{Iu}{X^{1,\frac{1}{2},1}_S}^4$.
\end{proof}

%%%%%%%%%%%%%%%%%%%%%%%%%%%%%%%%%%%%

\medskip
\subsection{Modified local well-posedness}

We begin with the following bilinear estimates.
\begin{lem}\label{lem_modifiedbe_R2}
Let $1>s>\frac{1}{2}$, $0\ge r\ge s-1$.
Then, we have
\begin{gather}
\norm{\Sc{I}_S(I^S_{s}(n_{\pm}u))}{X^{1,\frac{1}{2},1}_S(\de )}\lec \de ^{\frac{1+r}{2}-}\norm{Iu}{X^{1,\frac{1}{2},1}_S(\de )}\norm{In_{\pm}}{X^{0,\frac{1}{2},1}_{W_{\pm}}(\de )},\label{est_modifiedbe_R2_s}\\
\norm{\Sc{I}_{W_+}(|\nabla |I^{W_+}_r(u_1\bbar{u_2}))}{X^{0,\frac{1}{2},1}_{W_{\pm}}(\de )}\lec \de ^{\frac{1}{2}-}\norm{Iu_1}{X^{1,\frac{1}{2},1}_S(\de )}\norm{Iu_2}{X^{1,\frac{1}{2},1}_S(\de )}.\label{est_modifiedbe_R2_w}
\end{gather}
\end{lem}

\begin{proof}
\eqref{est_modifiedbe_R2_w} follows easily from \eqref{est_be_w}, $I^{W_+}_r\le 1$, and $\tnorm{u}{X^{s,\frac{1}{2},1}_S}\le \tnorm{Iu}{X^{1,\frac{1}{2},1}_S}$.
We thus focus on \eqref{est_modifiedbe_R2_s}.
First of all, we show
\eq{est_be_R2}{\norm{\Sc{I}_S(n_{\pm}u)}{X^{s,\frac{1}{2},1}_S(\de )}\lec \de ^{\frac{1+r}{2}-}\norm{u}{X^{s,\frac{1}{2},1}_S(\de )}\norm{n_{\pm}}{X^{r,\frac{1}{2},1}_{W_{\pm}}(\de )}.}
% The case $r=0$ follows from \eqref{est_be_s}, so assume $r<0$.
Use $\ze _0$, $\ze _1$, $\ze _2$ for the Fourier variables of $n_\pm$, $n_\pm u$, $u$, respectively (thus $\ze _0=\ze _1-\ze _2$).

(i) The case $|\xi _1|\lec |\xi _2|$.
Since $s\ge r+\frac{1}{2}$, \eqref{est_be_R2} is reduced to
\eqq{\norm{\Sc{I}_S(n_{\pm}u)}{X^{r+\frac{1}{2},\frac{1}{2},1}_S(\de )}\lec \de ^{\frac{1+r}{2}-}\norm{u}{X^{r+\frac{1}{2},\frac{1}{2},1}_S(\de )}\norm{n_{\pm}}{X^{r,\frac{1}{2},1}_{W_{\pm}}(\de )}.}
It is not difficult to obtain this by interpolation between \eqref{est_endpoint_R2} and \eqref{est_be_s} with $s=\frac{1}{2}$.

(ii) The case $|\xi _1|\gg |\xi _2|$.
An interpolation between Lemmas~\ref{prop_te_highlow} and \ref{prop_te_highlow2} implies
\eqq{\iint _{\ze _0=\ze _1-\ze _2}f(\ze _0)g_1(\ze _1)g_2(\ze _2)\lec L_{\max}^{\frac{1}{2}}(L_{\med}L_{\min})^{\frac{1-r}{4}+}N_2^{1+r-}N_1^{-1}\norm{f}{L^2}\norm{g_1}{L^2}\norm{g_2}{L^2}}
for $f,g_1,g_2\in L^2_{\ze}(\R \times \R^2)$ with 
\eqq{\Supp{f}{\FR{P}_{N_0}\cap \FR{W}_{L_0}^\pm},\quad \Supp{g_j}{\FR{P}_{N_j}\cap \FR{S}_{L_j}},\quad j=1,2,\quad N_1\gg N_2.}
(We can choose $1+r->\frac{1}{2}$ because $r>-\frac{1}{2}$.
Note that $L_{\max}\gec N_1^2$ is required for nonzero contribution under this assumption.)
To apply this, we have to decompose $\Sc{I}_S(n_{\pm}u)$ as
\eqq{\sum _{N_1\ge 1}\sum _{N_2\ll N_1}\sum _{N_0\sim N_1}\sum _{L_0,L_1,L_2\ge 1}\Sc{I}_SP^S_{N_1,L_1}(P^{W_{\pm}}_{N_0,L_0}n_{\pm}P^S_{N_2,L_2}u).}
If $L_0=L_{\max}$ (similar for the case $L_2=L_{\max}$), we use the above estimate and Lemma~\ref{lem_linear} to obtain
\eqq{&\norm{P_{N_1}\Sc{I}_S(P_{N_0}n_{\pm}\cdot u)}{X^{s,\frac{1}{2},1}_S(\de )}\lec \de ^{\frac{1+r}{4}-}N_1^s\sum _{N_2\ll N_1}\norm{P_{N_1}(P_{N_0}n_{\pm}\cdot \psi _\de P_{N_2}u)}{X^{0,-\frac{1-r}{4}-,\I}_S}\\
&\lec \de ^{\frac{1+r}{4}-}N_1^{s-1}N_0^{-r}\sum _{N_2\ll N_1}N_2^{1+r-s-}\norm{\psi _\de P_{N_2}u}{X^{s,\frac{1-r}{4}+,1}_S}\norm{P_{N_0}n_{\pm}}{X^{r,\frac{1}{2},1}_{W_{\pm}}}\\
&\lec \de ^{\frac{1+r}{2}-}\norm{u}{X^{s,\frac{1}{2},1}_S}\norm{P_{N_0}n_{\pm}}{X^{r,\frac{1}{2},1}_{W_{\pm}}},
}
where at the last inequality we have used the assumption $1+r-s\ge 0$.
Squaring and summing up the above in $N_1$ we obtain \eqref{est_be_R2} (note that $N_0\sim N_1$).
In the case $L_1=L_{\max}$, a similar argument yields
\eqq{&\norm{P_{N_1}\Sc{I}_S(P_{N_0}n_{\pm}\cdot u)}{X^{s,\frac{1}{2},1}_S(\de )}\lec N_1^s\sum _{L_1}L_1^{-\frac{1}{2}}\sum _{N_2\ll N_1}\norm{P^S_{N_1,L_1}(\psi _\de P_{N_0}n_{\pm}\cdot \psi _\de P_{N_2}u)}{L^2_{t,x}}\\
&\lec N_1^{s-1}N_0^{-r}\sum _{L_0,L_1,L_2}(L_0L_2)^{\frac{1-r}{4}+}\sum _{N_2\ll N_1}N_2^{1+r-s-}\cdot N_2^s\norm{\psi _\de P^S_{N_2,L_2}u}{L^2_{t,x}}\cdot N_0^r\norm{\psi _\de P^{W_{\pm}}_{N_0,L_0}n_{\pm}}{L^2_{t,x}}.
}
We can carry out the sum in $L_1$ using the fact $L_1\sim \max \shugo{\bbar{L}_{02},\,N_1^2}$, and have the same bound as the previous case.
This completes the proof of \eqref{est_be_R2}.

To upgrade \eqref{est_be_R2} to \eqref{est_modifiedbe_R2_s}, we only have to show
\eqq{m_{1-s,N}(\xi _1)\LR{\xi _1}^{1-s}\lec m_{1-s,N}(\xi _2)\LR{\xi _2}^{1-s}\cdot m_{-r,N}(\xi _0)\LR{\xi _0}^{-r}}
for $\xi _0,\xi _1,\xi _2$ such that $\xi _0=\xi _1-\xi _2$.
This is true for the case $|\xi _1|\lec |\xi _2|$ or the case $|\xi _2|\gec N$, because if $q\ge 0$ we have $m_{q,N}(\xi )\LR{\xi}^q\ge 1$, $m_{q,N}(\xi _1)\LR{\xi_1}^q\lec m_{q,N}(\xi _2)\LR{\xi_2}^q$ for $|\xi _1|\lec |\xi _2|$, and $m_{q,N}(\xi )\LR{\xi}^q\sim m_{q,N}(\xi )|\xi |^q=N^q$ for $|\xi |\ge 2N$.

In the remaining case, $|\xi _2|\ll |\xi _1|$ and $|\xi _2|\ll N$, we have $|\xi _0|\sim |\xi _1|$ and then 
\eqq{m_{1-s,N}(\xi _1)\sim m_{1-s,N}(\xi _2)m_{1-s,N}(\xi _0)\lec m_{1-s,N}(\xi _2)m_{-r,N}(\xi _0),}
since $1-s\ge -r$.
This and \eqref{est_be_s} with $s=1$ imply \eqref{est_modifiedbe_R2_s}.
\end{proof}

By a standard argument, we can deduce from Lemma~\ref{lem_modifiedbe_R2} the following local well-posedness.
\begin{prop}\label{prop_modifiedlwp_R2}
Let $1>s>\frac{1}{2}$, $0\ge r\ge s-1$.
Then, for any $(u_0,n_{+0})\in H^s\times H^r$, there exists a unique solution to \eqref{ZH2} on $\R^2$, $(u,n_+)\in X^{s,\frac{1}{2},1}_S(\de )\times X^{r,\frac{1}{2},1}_{W_+}(\de )$, with the existence time
\eqq{\de \sim (\norm{Iu_0}{H^1}+\norm{In_{+0}}{L^2})^{-\frac{2}{1+r}-},}
such that the following estimate holds:
\eqq{\norm{Iu}{X^{1,\frac{1}{2},1}_S(\de )}+\norm{In_+}{X^{0,\frac{1}{2},1}_{W_+}(\de )}\lec \norm{Iu_0}{H^1}+\norm{In_{+0}}{L^2}.}
In particular, we have
\eqq{\sup _{-\de \le t\le \de}\big( \norm{Iu(t)}{H^1}+\norm{In_+(t)}{L^2}\big) \lec \norm{Iu_0}{H^1}+\norm{In_{+0}}{L^2}.}
\end{prop}

We remark that our local existence time $\de \sim \tnorm{\text{data}}{}^{-\frac{2}{1+r}-}$ is longer than that obtained in \cite{P12}, which was $\de \sim \tnorm{\text{data}}{}^{-\frac{2}{1+2r}-}$.
Compare the bilinear estimate \eqref{est_modifiedbe_R2_s} with Lemma~2.1 in \cite{P12}.
In fact, a longer local existence time will lead to the global well-posedness for a lower regularity.

\medskip
\subsection{Proof of Theorem~\ref{thm_global_R2}}

Here we assume
\begin{gather}
1>s>\tfrac{1}{2},\qquad 0\ge r\ge s-1,\label{assumption1}\\
r>1-2s,\qquad r>-\tfrac{1}{2}s.\label{assumption2}
\end{gather}
Let $(u_0,n_{0+})\in H^s \times H^r$ be an initial datum with $\tnorm{u_0}{L^2}<\tnorm{Q}{L^2(\R^2)}$.
The modified energy $H(Iu,In_+)(t)$, satisfying the initial bound
\eqq{H(Iu_0,In_{+0})\le C(N^{2(1-s)}+N^{-2r})\le C_0N^{2(1-s)},}
controls $\tnorm{Iu(t)}{H^1}+\tnorm{In_+(t)}{L^2}$.
Proposition~\ref{prop_modifiedlwp_R2} shows that the solution on $[0,t_0]$ can be extended up to $t=t_0+\de$ with a uniform time $\de \sim N^{-\frac{2(1-s)}{1+r}-}$ and satisfies
\eqq{\norm{Iu(\cdot -t_0)}{X^{1,\frac{1}{2},1}_S(\de )}+\norm{In_+(\cdot -t_0)}{X^{0,\frac{1}{2},1}_{W_+}(\de )}\lec N^{1-s},}
as long as $H(Iu,n_{+})(t_0)\le 2C_0N^{2(1-s)}$.
If we could iterate the local theory $M$ times, then from Propositions~\ref{prop_fixedtime} and \ref{prop_ac_R2},
\eqq{&|H(Iu,n_+)(M\de )-H(Iu,n_+)(0)|\\
&\lec N^{-1+}(N^{1-s})^3+M\Big\{ N^{-1+}\de ^{\frac{1}{2}-}(N^{1-s})^3+(N^{-2+}+N^{-\frac{5}{4}+}\de ^{\frac{1}{4}-}+N^{-1+}\de ^{1-})(N^{1-s})^4\Big\} \\
&\sim \Big\{ N^{-s+}+MN^{-\al _0(s,r)+}\Big\} N^{2(1-s)},\qquad \al _0(s,r):=\min \shugo{\tfrac{1+rs}{1+r},\, \tfrac{-1-3r+6s+8rs}{4(1+r)}}.}
Thus, we can repeat the local procedure $O(N^{\al _0-})$ times to reach some time $\sim \de N^{\al _0-}\sim N^{\al _1-}$,
\eqq{\al _1(s,r):=\min \shugo{\tfrac{-1+2s+rs}{1+r},\, \tfrac{-9-3r+14s+8rs}{4(1+r)}}.}
The required conditions for global well-posedness are
\begin{gather}
-1+2s+rs>0,\label{assumption3}\\
-9-3r+14s+8rs>0.\label{assumption4}
\end{gather}
It turns out that \eqref{assumption2} and \eqref{assumption3} are automatically satisfied under the assumptions \eqref{assumption1} and \eqref{assumption4}.
Moreover, we have
\eqs{\sup _{-T\le t\le T}\big( \norm{Iu(t)}{H^1}+\norm{In_+(t)}{L^2}\big) \lec N^{1-s}\sim T^{\frac{1-s}{\al _1}+}\sim T^{\al _2+},\\
\al _2(s,r):=\max \shugo{\tfrac{(1-s)(1+r)}{-1+2s+rs},\, \tfrac{4(1-s)(1+r)}{-9-3r+14s+8rs}}.}
We obtain the same a priori estimate for solutions to the original equation \eqref{ZH}, concluding the proof of Theorem~\ref{thm_global_R2}.

%%%%%%%%%%%%%%%%%%%%%%%%%%%%%%%%
%%%%%%%%%%%%%%%%%%%%%%%%%%%%%%%%
%%%%%%%%%%%%%%%%%%%%%%%%%%%%%%%%

\bigskip
\appendix
\section{Proof of Lemma~\ref{lem_BEforwave}}

Here we shall give a proof of the following bilinear estimate.
\begin{prop}\label{prop_saigo}
We have
\eqq{\norm{uv}{L^2_{t,x}}\lec L^{\frac{3}{4}}N^{\frac{3}{4}}\norm{u}{L^2_{t,x}}\norm{v}{L^2_{t,x}}}
for $u,v\in L^2(\R \times Z)$, $Z=\Tg^2$ or $\R^2$, such that $\mathrm{supp}~\ti{u},\Supp{\ti{v}}{\FR{P}_{N}\cap \FR{W}^{+}_{L}}$.
\end{prop}
Lemma~\ref{lem_BEforwave} then follows by letting $v=u$.
The standard argument reduces the problem to the following; for details, see \mbox{e.g.} the proof of Lemma~2.5 in \cite{K11}.
\begin{prop}
Let $N,L\ge 1$.
Then, for any $k\in \R^2$ and $A\ge |k|$, the set
\eqq{\Shugo{k'\in \R^2}{|k'|\le N,\,|k-k'|\le N,\,|k'|+|k-k'|\in [A,A+L]}}
is covered with at most $O(N^{\frac{3}{2}}L^{\frac{1}{2}})$ squares of unit size.
\end{prop}

We begin with preparing the following lemma.
\begin{lem}\label{lem_ellipse}
Let $a\ge b\gg 1$.
Define
\eqq{E_<&:=\Shugo{(x,y)\in \R^2}{\frac{x^2}{a^2}+\frac{y^2}{b^2}\le 1},\\
E_>&:=\Shugo{(x,y)\in \R^2}{\frac{x^2}{(a+100\frac{a}{b})^2}+\frac{y^2}{(b+100)^2}\ge 1}.}
Then, there exists no unit square in $\R^2$ intersecting with both $E_<$ and $E_>$.
The same holds for 
\eqq{E'_<&=\Shugo{(x,y)\in \R^2}{\frac{x^2}{(a-100\frac{a}{b})^2}+\frac{y^2}{(b-100)^2}\le 1},\\
E'_>&=\Shugo{(x,y)\in \R^2}{\frac{x^2}{a^2}+\frac{y^2}{b^2}\ge 1}}
instead of $E_<$, $E_>$.
\end{lem}
\begin{proof}
We only prove the first half of the claim.
The second half will be shown by a similar argument.

Assume for contradiction that there existed such a square of side length $1$.
Then, it would hold for some $(x,y)\in E_<$ and $(x',y')\in E_>$ that
\begin{gather}
(x-x')^2+(y-y')^2\le 2,\label{1}\\
\frac{x^2}{a^2}+\frac{y^2}{b^2}\le \frac{x'^2}{(a+100\frac{a}{b})^2}+\frac{y'^2}{(b+100)^2}\label{2}.
\end{gather}
Note that
\eqq{\frac{x'^2}{(a+100\frac{a}{b})^2}-\frac{x^2}{a^2}&=\frac{x'^2-x^2}{a^2}-x'^2\Big( \frac{1}{a^2}-\frac{1}{(a+100\frac{a}{b})^2}\Big) \\
&=\frac{x'+x}{a^2}(x'-x)-\frac{x'^2}{(a+100\frac{a}{b})^2}\Big( \frac{200}{b}+(\frac{100}{b})^2\Big) ,\\
\frac{y'^2}{(b+100)^2}-\frac{y^2}{b^2}&=\frac{y'^2-y^2}{b^2}-y'^2\Big( \frac{1}{b^2}-\frac{1}{(b+100)^2}\Big) \\
&=\frac{y'+y}{b^2}(y'-y)-\frac{y'^2}{(b+100)^2}\Big( \frac{200}{b}+(\frac{100}{b})^2\Big) .}
From these estimates and the fact $(x',y')\in E_>$,
\begin{alignat*}{2}
\frac{x'^2}{(a+100\frac{a}{b})^2}&+\frac{y'^2}{(b+100)^2}-\frac{x^2}{a^2}-\frac{y^2}{b^2}\\
&\le \frac{|x'+x|}{a^2}|x'-x|+\frac{|y'+y|}{b^2}|y'-y|-\Big( \frac{200}{b}+(\frac{100}{b})^2\Big) ,
\intertext{which is, from \eqref{1} and $(x,y)\in E_<$,}
&\le \frac{2|x|+\sqrt{2}}{a^2}\sqrt{2}+\frac{2|y|+\sqrt{2}}{b^2}\sqrt{2}-\frac{200}{b}\\
&\le \frac{10}{a}+\frac{10}{b}-\frac{200}{b}\le -\frac{180}{b}<0.
\end{alignat*}
This contradicts \eqref{2}.
\end{proof}

\begin{proof}[Proof of Proposition~\ref{prop_saigo}]
We may assume $|k|\le 2N$, otherwise the set is empty.
Treat several cases separately.

(i) $L\gec N$.
In this case, we use the condition $|k'|\le N$ to estimate the number of squares by $N^2\lec N^{\frac{3}{2}}L^{\frac{1}{2}}$.

(ii) $L\ll N$, $|k|\lec 1$.
In this case we have $|k'|\le N$ and $A-C\le 2|k'|\le A+L+C$.
It is easy to see that such a region, which is a disk of radius $L$ or the intersection of a disk of radius $N$ and an annulus of width $L$, can be covered with $\lec NL$ unit squares.
$L\lec N$ implies the claim.

(iii) $L\ll N$, $A\le |k|+10L$.
We have 
\eqq{|k'|+|k-k'|\le |k|+11L,}
which shows that $k'$ is inside an ellipse of distance between foci $|k|$, length of long axis
\eqq{|k|+11L\lec N,}
and length of short axis
\eqq{\sqrt{(|k|+11L)^2-|k|^2}=\sqrt{22|k|L+121L^2}\lec \sqrt{NL}.}
Therefore, we can cover this region with $\lec N\times \sqrt{NL}$ unit squares.

We remark that $k'$ is confined to the region 
\eqq{\Sc{R}:=\Shugo{k'\in \R^2}{|k'|+|k-k'|\in [A,A+L]}}
between two ellipses with common foci $0$, $k$, longer axis $A$ and $A+L$, respectively.

(iv) $L\ll N$, $A\ge 10N$.
In this case the region is close to an annulus.
In fact, 
\begin{alignat*}{4}
2a&=A,&\qquad 2b&=\sqrt{A^2-|k|^2}\ge \sqrt{A^2-(A/5)^2}\ge \tfrac{9}{10}\cdot 2a,\\
2a'&=A+L,&2b'&=\sqrt{(A+L)^2-|k|^2},
\end{alignat*}
with $2a,2a'$ (\mbox{resp.} $2b, 2b'$) the length of the long (\mbox{resp.} short) axes of inner and outer ellipses.
We first change the scale in the direction of short axis to make the inner ellipse a circle.
Then, the new region $\Sc{R}'$ is included in an annulus of width $\max \shugo{a'-a, \frac{a}{b}(b'-b)}$.
We see $a'-a=L$ and
\eqq{2\frac{a}{b}(b-b')&=\frac{a}{b}(\sqrt{(A+L)^2-|k|^2}-\sqrt{A^2-|k|^2})\\
&=\frac{a}{b}\frac{2AL+L^2}{\sqrt{(A+L)^2-|k|^2}+\sqrt{A^2-|k|^2}}\sim 1\cdot \frac{AL}{A}=L.}
Hence, the intersection of any ball of radius $2N$ and $\Sc{R}'$ is covered with $\lec NL$ unit squares, which shows that the intersection of any ball of radius $N$ and the original $\Sc{R}$ is also covered with the same number of unit squares.

(v) $L\ll N$, $|k|\gg 1$, and $|k|+10L\le A\le 10N$.
By translation and rotation, we may consider the covering of
\eqq{\ti{\Sc{R}}:=\Shugo{(x,y)\in \R^2}{\frac{x^2}{a'^2}+\frac{y^2}{b'^2}\le 1\le \frac{x^2}{a^2}+\frac{y^2}{b^2}}}
with $2a=A$, $2a'=A+L$, $2b=\sqrt{A^2-|k|^2}$, $2b'=\sqrt{(A+L)^2-|k|^2}$.
Note also that
\eqq{a\ge \tfrac{1}{2}|k|\gg 1,\qquad b=\tfrac{1}{2}\sqrt{A^2-|k|^2}\ge \tfrac{1}{2}\sqrt{(|k|+10L)^2-|k|^2}\ge \sqrt{|k|L}\gg 1.}
From Lemma~\ref{lem_ellipse}, we see that the smallest (axis-aligned) lattice polygon including the inside of outer boundary of $\ti{\Sc{R}}$ is included in the inside of an ellipse with long axis $2(a'+100\frac{a'}{b'})$ and short axis $2(b'+100)$.
In the same manner, the biggest (axis-aligned) lattice polygon included in the inside of inner boundary of $\ti{\Sc{R}}$ includes an ellipse with long axis $2(a-100\frac{a}{b})$ and short axis $2(b-100)$.
Therefore, the number of needed unit squares is estimated by
\eqq{&\Big( a'+100\frac{a'}{b'}\Big) (b'+100)-\Big( a-100\frac{a}{b}\Big) (b-100)\\
&=\Big( (a'-a)+100(\frac{a'}{b'}+\frac{a}{b})\Big) (b'+100)~+~\Big( a-100\frac{a}{b}\Big) (b'-b+200).}
We find $b'+100\lec N$, $|a-100\frac{a}{b}|\lec N$, $a'-a\lec L$, and
\eqq{\frac{a}{b}&=\frac{A}{\sqrt{A^2-|k|^2}}=\frac{1}{\sqrt{1-(\frac{|k|}{A})^2}}\le \frac{1}{\sqrt{1-(\frac{|k|}{|k|+10L})^2}}\le \frac{1}{\sqrt{1-(\frac{2N}{2N+10L})^2}}\\
&=\frac{2N+10L}{\sqrt{(2N+10L)^2-(2N)^2}}\sim \frac{N}{\sqrt{NL}}=\sqrt{\dfrac{N}{L}}.}
We also see $a'/b'\lec \sqrt{N/L}$ in the same manner.
Finally, 
\eqq{2(b'-b)&=\sqrt{(A+L)^2-|k|^2}-\sqrt{A^2-|k|^2}=\frac{2AL+L^2}{\sqrt{(A+L)^2-|k|^2}+\sqrt{A^2-|k|^2}}\\
&\lec \frac{A}{\sqrt{A^2-|k|^2}}L\lec \sqrt{\dfrac{N}{L}}L=\sqrt{NL}.}
With all of them together, we reach the bound $\lec N^{\frac{3}{2}}L^{\frac{1}{2}}$.
\end{proof}

%%%%%%%%%%%%%%%%%%%%%%%%%%%%%%%%
%%%%%%%%%%%%%%%%%%%%%%%%%%%%%%%%
%%%%%%%%%%%%%%%%%%%%%%%%%%%%%%%%

\bigskip
\section*{Acknowledgments}

The author thanks Takamori Kato for reading an earlier version of the manuscript and giving a shorter proof.
This work was partially supported by Grant-in-Aid for JSPS Fellows 08J02196.

%%%%%%%%%%%%%%%%%%%%%%%%%%%%%%%%%%%
%%%%%%%%%%%%%%%%%%%%%%%%%%%%%%%%%%%
%%%%%%%%%%%%%%%%%%%%%%%%%%%%%%%%%%%

\end{document}